\documentclass[11pt,reqno]{amsart}
\usepackage{etex}
\usepackage{amsmath}
\usepackage{anysize}
\usepackage{mathrsfs}
\usepackage{yhmath}
\usepackage{amsxtra}
\usepackage{amscd}
\usepackage{amsthm}
\usepackage{amsfonts}
\usepackage{amssymb}
\usepackage{bbm}
\usepackage{eucal}
\usepackage{enumerate}
\usepackage{mathtools}
\usepackage{pmgraph}
\usepackage{stmaryrd}
\usepackage[usenames,dvipsnames]{color}
\usepackage{xcolor}
\usepackage[bookmarks=true]{hyperref}

\marginsize{3cm}{3cm}{3cm}{3cm}

\input prepictex
\input pictex
\input postpictex

\newtheorem{thm}{Theorem}[section]
\newtheorem{cor}[thm]{Corollary}
\newtheorem{lem}[thm]{Lemma}
\newtheorem{prop}[thm]{Proposition}

\theoremstyle{definition}

\theoremstyle{remark}
\newtheorem{rem}[thm]{Remark}

\numberwithin{equation}{section}

\begin{document}

\newcommand{\thmref}[1]{Theorem~\ref{#1}}
\newcommand{\secref}[1]{Section~\ref{#1}}
\newcommand{\lemref}[1]{Lemma~\ref{#1}}
\newcommand{\propref}[1]{Proposition~\ref{#1}}
\newcommand{\corref}[1]{Corollary~\ref{#1}}
\newcommand{\remref}[1]{Remark~\ref{#1}}
\newcommand{\eqnref}[1]{(\ref{#1})}
\newcommand{\exref}[1]{Example~\ref{#1}}

\DeclarePairedDelimiterX\setc[2]{\{}{\}}{\,#1 \;\delimsize\vert\; #2\,}

\newcommand{\nc}{\newcommand}
\nc{\Z}{{\mathbb Z}}
\nc{\C}{{\mathbb C}}
\nc{\N}{{\mathbb N}}
\nc{\Zp}{{\mathbb Z}_+}
\nc{\Zs}{{\mathbb Z}^*}
\nc{\OO}{\mc O}

\nc{\la}{\lambda}
\nc{\ep}{\epsilon}
\nc{\mc}{\mathcal}
\nc{\mf}{\mathfrak}
\nc{\h}{\mf h}
\nc{\cL}{{\mc L}}
\nc{\cH}{{\mc H}}
\nc{\G}{{\mf g}}
\nc{\DG}{\widetilde{\mf g}}
\nc{\SG}{\overline{\mf g}}
\nc{\La}{\Lambda}
\nc{\V}{\mf V} \nc{\bi}{\bibitem}

\nc{\hf}{\frac{1}{2}}
\nc{\gl}{{\mf{gl}}}
\nc{\osp}{\mf{osp}}
\nc{\spo}{\mf{spo}}
\nc{\hosp}{\widehat{\mf{osp}}}
\nc{\hspo}{\widehat{\mf{spo}}}
\nc{\I}{\mathbb{I}}
\nc{\J}{\mathbb{J}}
\nc{\X}{\mathbb{X}}
\nc{\hh}{\widehat{\mf{h}}}
\nc{\w}{\widetilde}
\nc{\ov}{\overline}
\nc{\un}{\underline}

\nc{\wla}{\w{\la}}
\nc{\ovla}{\ov{\la}}
\nc{\wmu}{\w{\mu}}
\nc{\ovmu}{\ov{\mu}}
\nc{\fh}{\mf h}
\nc{\ovfh}{\ov{{\mf h}}}
\nc{\wfh}{\tilde{{\mf h}}}
\nc{\cP}{{\mathcal{P}}}
\nc{\ovcH}{\ov{{\mathcal{H}}}}
\nc{\ovtH}{\ov{\texttt{H}}}
\nc{\wcH}{\w{\mathcal{H}}}
\nc{\cG}{\mathcal{G}}
\nc{\ovcG}{\ov{\mathcal{G}}}
\nc{\wcG}{\w{\mathcal{G}}}
\nc{\cQ}{\mathcal{Q}}
\nc{\cS}{\mathcal{S}}

\nc{\cC}{{\mathcal{C}}}
\nc{\ovcC}{\ov{{\mathcal{C}}}}
\nc{\wcC}{\w{\mathcal{C}}}
\nc{\cD}{{\mathcal{D}}}
\nc{\ovcD}{\ov{{\mathcal{D}}}}
\nc{\wcD}{\w{\mathcal{D}}}

\nc{\ft}{{\mf{t}}}
\nc{\ze}{{(0)}}
\nc{\mth}{\mathring{\mf h}}
\nc{\mtb}{\mathring{\mf b}}
\nc{\mtp}{\mathring{\mf p}}
\nc{\mtl}{\mathring{\mf l}}
\nc{\mwh}{\mathring{\w{\mf h}}}
\nc{\mwb}{\mathring{\w{\mf b}}}
\nc{\mwp}{\mathring{\w{\mf p}}}
\nc{\mwl}{\mathring{\w{\mf l}}}
\nc{\msh}{\mathring{\ov{\mf h}}}
\nc{\msb}{\mathring{\ov{\mf b}}}
\nc{\msp}{\mathring{\ov{\mf p}}}
\nc{\msl}{\mathring{\ov{\mf l}}}

\advance\headheight by 2pt

\title[Gaudin Hamiltonians and super KZ equations]
{Quadratic and cubic Gaudin Hamiltonians and super Knizhnik-Zamolodchikov equations for general linear Lie superalgebras}

\author[B. Cao]{Bintao Cao}
\address{School of Mathematics and Statistics, Yunnan University, Kunming, China, 650500}
\email{btcao@ynu.edu.cn}

\author[W. K. Cheong]{Wan Keng Cheong}
\address{Department of Mathematics, National Cheng Kung University, Tainan, Taiwan 701401}
\email{keng@ncku.edu.tw}

\author[N. Lam]{Ngau Lam}
\address{Department of Mathematics, National Cheng Kung University, Tainan, Taiwan 701401}
\email{nlam@ncku.edu.tw}

 \begin{abstract}

We show that under a generic condition, the quadratic Gaudin Hamiltonians associated to $\mathfrak{gl}(p+m|q+n)$ are diagonalizable on any singular weight space in any tensor product of unitarizable highest weight $\mathfrak{gl}(p+m|q+n)$-modules. Moreover, every joint eigenbasis of the Hamiltonians can be obtained from some joint eigenbasis of the quadratic Gaudin Hamiltonians for the general linear Lie algebra $\mathfrak{gl}(r+k)$ on the corresponding singular weight space in the tensor product of some finite-dimensional irreducible $\mathfrak{gl}(r+ k)$-modules for $r$ and $k$ sufficiently large. After specializing to $p=q=0$, we show that similar results hold as well for the cubic Gaudin Hamiltonians associated to $\mathfrak{gl}(m|n)$.

We also relate the set of singular solutions of the (super) Knizhnik-Zamolodchikov equations for $\mathfrak{gl}(p+m|q+n)$ to the set of singular solutions of the Knizhnik-Zamolodchikov equations for $\mathfrak{gl}(r+k)$ for $r$ and $k$ sufficiently large.

\end{abstract}

 \maketitle

\setcounter{tocdepth}{1}

\section{Introduction}

The Gaudin model associated to any finite-dimensional simple Lie algebra was originally introduced by Gaudin in \cite{G76, G83}. One of the main problems in the Gaudin model is the simultaneous diagonalization of (higher) Gaudin Hamiltonians (cf. \cite{BF, FFR, MTV06, MTV09, FFRyb, Ryb}). There are also several works, e.g.,  \cite{KuM, MVY, ChL}, which focus on the Gaudin Hamiltonians for classical Lie superalgebras.

One of the goals of this paper is to extend the results of \cite{ChL} to the super version of the quadratic and cubic Gaudin Hamiltonians associated to general linear Lie superalgebras.

Fix nonnegative integers $p, q, m, n$ with $n \geq 1$.
Let $\ovcG[q,m]_{(p,n)}$ and $\cG_{(p,n)}$ be the general linear Lie (super)algebras introduced in Section \ref{gl}. Note that $\ovcG[q,m]_{(p,n)}\cong \gl(p+m|q+n)$ and $\cG_{(p,n)} \cong \gl(p+n)$.

Suppose $\ell$ is a positive integer with $\ell \geq 2$. For $i=1, \ldots, \ell$, we study the quadratic Gaudin Hamiltonians $\ov H^i[q,m]_{(p,n)}$ associated to $\ovcG[q,m]_{(p,n)}$ and $H^i_{(p,n)}$  associated to $\cG_{(p,n)}$; see \eqref{qGH}.
We also study the cubic Gaudin Hamiltonians $\ovcC^i[m]_n$ and $\ovcD^i[m]_n$ associated to $\gl(m |n)$ and the cubic Gaudin Hamiltonians $\cC^i_n$ and $\cD^i_n$ associated to $\gl(n)$; see \eqref{cGH1} and \eqref{cGH2}. All of these Hamiltonians depend on $(z_1, \ldots, z_\ell)\in \C^\ell$, where $z_i$'s are pairwise distinct.

The proofs of the simultaneous diagonalization of (higher) Gaudin Hamiltonians on the tensor product of finite-dimensional irreducible modules over a simple Lie algebra involve certain geometric arguments together with the fact that the property of Gaudin algebras having simple joint spectra is an open condition on the parameters $(z_1, \ldots, z_\ell)\in \C^\ell$ (cf. \cite{MTV09,Ryb}).

We apply super duality, which is discussed in Section \ref{SD}, to relate Gaudin Hamiltonians associated to Lie superalgebras to those associated to Lie algebras and demonstrate the following result on quadratic Gaudin Hamiltonians.

\begin{thm} [\thmref{thm:quad-diag}] \label{main-1}
Let $\ov L^\otimes=\ov{L}_1 \otimes \cdots \otimes \ov{L}_\ell$ be a tensor product of unitarizable highest weight $\ovcG[q,m]_{(p,n)}$-modules. (For $m=p=0$, $\ov L_i$ are infinite-dimensional unitarizable highest weight modules over $\mf{gl}(q+n)$).
For a generic $(z_1, \ldots, z_\ell)\in \C^\ell$,
the quadratic Gaudin Hamiltonians
$\ov H^1[q,m]_{(p,n)},\ldots, \ov H^{\ell}[q,m]_{(p,n)}$
are simultaneously diagonalizable on any singular weight space of $\ov L^\otimes$,
and each joint eigenbasis of the Hamiltonians can be obtained from some joint eigenbasis of the quadratic Gaudin Hamiltonians
$H^1_{(r, k)},\ldots,H^{\ell}_{(r, k)}$ on the corresponding singular weight space in the tensor product of the corresponding finite-dimensional irreducible $\cG_{(r, k)}$-modules for $r$ and $k$ sufficiently large.
\end{thm}

By taking $p=q=0$, \thmref{main-1} specializes to the diagonalization of the quadratic Gaudin Hamiltonians for $\gl(m|n)$ on the tensor product of irreducible polynomial modules, which was studied previously by Mukhin, Vicedo and Young \cite{MVY}. We also show that the action of the algebra generated by the Hamiltonians $\ov H^1[0,m]_{(0,n)},\ldots, \ov H^{\ell}[0,m]_{(0,n)}$ on the singular space of the tensor power of the natural module $\C^{m|n}$ contains a cyclic vector for each $(z_1, \ldots, z_\ell)\in \C^\ell$, where $z_i$'s are pairwise distinct (see \thmref{thm:cyclic}).

Furthermore, we establish the following result on the simultaneous diagonalization of the cubic Gaudin Hamiltonians.

\begin{thm} [\thmref{thm:cubic-diag}] \label{main-2}
Let $\ov L^\otimes=\ov{L}_1 \otimes \cdots \otimes \ov{L}_\ell$ be a tensor product of irreducible polynomial $\gl(m|n)$-modules.
For a generic $(z_1, \ldots, z_\ell)\in \C^\ell$,
 the cubic Gaudin Hamiltonians $\ovcC^1[m]_{n},\ldots,\ovcC^\ell[m]_{n}$
and
$\ovcD^1[m]_{n},\ldots,\ovcD^\ell[m]_{n}$
are simultaneously diagonalizable on any singular weight space of $\ov L^\otimes$,
and each joint eigenbasis of the Hamiltonians can be obtained from some joint eigenbasis of the quadratic Gaudin Hamiltonians for $\gl(k)$ on the corresponding singular weight space in the tensor product of the corresponding finite-dimensional irreducible $\gl(k)$-modules for $k$ sufficiently large.
\end{thm}

The Knizhnik-Zamolodchikov (KZ) equations for Lie algebras are a system of partial differential equations defined in terms of quadratic Gaudin Hamiltonians (cf. \cite{KZ, EFK, SV, V}). The super KZ equations for Lie superalgebras can be formulated in a similar way (see, e.g., \cite{CaL, GZZ, KuM}).

Another goal of this paper is to study (super) KZ equations and generalize the results of \cite{CaL} to the Lie superalgebra $\ovcG[q,m]_{(p,n)}$. Particularly, we have the following.

\begin{thm} [\thmref{iso-KZ}] \label{main-3}
There is a linear isomorphism from the set of singular solutions of the (super) KZ equations of any fixed weight on any tensor product of $\ovcG[q,m]_{(p,n)}$-modules in $\mathring{\ov\OO}[q,m]_{(p,n)}$ to the set of singular solutions of the KZ equations of the corresponding weight on the tensor product of the corresponding $\cG_{(r,k)}$-modules (nonsuper!) in $\mathring{\OO}{(r,k)}$ for $r$ and $k$ sufficiently large. (Here the categories $\mathring{\ov\OO}[q,m]_{(p,n)}$ and $\mathring{\OO}{(r,k)}$ are defined in Section \ref{sec:SKZ}.)
\end{thm}

This paper is organized as follows. In Section \ref{sec:pre}, we fix notation and discuss super duality relating the parabolic Bernstein--Gelfand--Gelfand (BGG) categories of modules over the central extensions of general linear Lie (super)algebras. We also introduce unitarizable modules, which play a special role in this paper. In Section \ref{sec:quad-G}, we investigate the (super) quadratic Gaudin Hamiltonians and the (super) KZ equations for the central extensions of finite-rank and infinite-rank general linear Lie (super)algebras. In Section \ref{sec:quad-cG}, we discuss the (super) quadratic Gaudin Hamiltonians and the (super) KZ equations for $\gl(p+m|q+n)$. Particularly, we prove \thmref{main-1} and \thmref{main-3}. In Section \ref{sec:cubic}, we study the (super) cubic Gaudin Hamiltonians for $\gl(m|n)$ and give a proof of \thmref{main-2}.

\vskip 0.3cm
\noindent{\bf Notations.} Throughout the paper, the symbols $\Z$, $\N$, $\Zp$, and $\Zs$
stand for the sets of all, positive, non-negative integers, and nonzero integers, respectively,
and the symbol $\C$ stands for the field of complex numbers.
All vector spaces, algebras, tensor products, etc., are over $\C$.
{\bf  We fix $q,m \in \Zp$, $p \in \Zp\cup\{\infty\}$ and $n \in \N\cup\{\infty\}$}.

\bigskip

\section{Preliminaries} \label{sec:pre}

In this section, we define the general linear Lie (super)algebras
$\wcG$, $\cG_{(p,n)}$ and $\ovcG[q,m]_{(p,n)}$
as well as their central extensions
$\DG$, $\G_{(p,n)}$ and ${\SG}[q,m]_{(p,n)}$.
We introduce the parabolic BGG category
$\w\OO_{(\infty,\infty)}$ (resp., $\OO_{(p,n)}$ and $\ov\OO[q,m]_{(p,n)}$)
of modules over $\DG$ (resp., $\G_{(p,n)}$ and ${\SG}[q,m]_{(p,n)}$)
and the truncation functor which relates the category
$\OO_{(\infty,\infty)}$ (resp., $\ov\OO[q,m]_{(\infty,\infty)}$)
for the infinite-rank Lie (super)algebra
$\G_{(\infty,\infty)}$ (resp., $\SG[q,m]_{(\infty,\infty)}$)
to the category
$\OO_{(p,n)}$ (resp., $\ov\OO[q,m]_{(p,n)}$)
for the finite-rank Lie (super)algebra
$\G_{(p,n)}$ (resp., $\SG[q,m]_{(p,n)}$).
We also discuss super duality which shows that the tensor functors
$T:\w\OO_{(\infty,\infty)} \to \OO_{(\infty,\infty)}$ and $\ov{T}_{[q,m]}: \w\OO_{(\infty,\infty)} \to \ov\OO[q,m]_{(\infty,\infty)}$
are equivalences of tensor categories. Finally, we describe unitarizable modules over $\ovcG[q,m]_{(p,n)}$.

\subsection{General linear (super)algebras} \label{gl}

Let $\w{V}$ denote the superspace over $\C$ with ordered basis $\{v_i \,|\, i \in \hf\Zs\}$,
where the parity of $v_i$ is given by $|v_i|:=\ov{2i} \in \Z_2$.

Let $\wcG:=\gl(\w{V})$ be the Lie superalgebra consisting of all linear endomorphisms on
$\w{V}$ which vanish on all but finitely many $v_i$'s.
For $i,j \in \hf\Zs$, we let $E_{i,j}$ be the linear endomorphism on $\w{V}$ defined by
$$E_{i,j}(v_r)=\delta_{j,r} v_i  \qquad \textrm{for } r \in \hf\Zs,$$
where $\delta$ is the Kronecker delta.
The set
$\{ E_{i,j} \, | \, i,j\in \hf\Zs\}$
is a basis for $\wcG$.
We let
$\w{\bf b}:=\bigoplus_{ i,j\in \hf\Zs, \, i\le j}\C E_{i,j}$
be the standard Borel subalgebra of $\wcG$.

Fix $q, m \in \Zp$, $p \in \Zp\cup\{\infty\}$ and $n \in \N\cup\{\infty\}$. Let
\begin{align*}
\w\I_{(p,n)}&=\setc*{i \in \hf \Zs}{-p \le i \le n},\\
\I_{(p,n)}&=\setc*{i \in \hf+\Z}{-p < i < n}, \\
\ov{\I}{[q, m]}_{(p, n)} &= \setc[\Big]{i \in \Zs}{-p \le i \le m } \cup \setc[\Big]{j \in \hf+\Z}{-q < j < n }.
\end{align*}
\textbf{We will drop the subscript $(p,n)$ if $(p,n)=(\infty, \infty)$.}
For instance,
$\ov{\I}{[q, m]}:=\ov{\I}{[q, m]}_{(\infty, \infty)}$.
We define the total order of
$\ov{\I}{[q, m]}$ as follows:
 \begin{align*}
\ldots&<_{\ov{\I}{[q, m]}}-2 <_{\ov{\I}{[q, m]}}-1 <_{\ov{\I}{[q, m]}}-(q-\hf)<_{\ov{\I}{[q, m]}}\ldots<_{\ov{\I}{[q, m]}}-\frac{3}{2}<_{\ov{\I}{[q, m]}}-\hf\\
&<_{\ov{\I}{[q, m]}} 1<_{\ov{\I}{[q, m]}}2 <_{\ov{\I}{[q, m]}} \ldots <_{\ov{\I}{[q, m]}}m <_{\ov{\I}{[q, m]}}\hf<_{\ov{\I}{[q, m]}}\frac{3}{2} <_{\ov{\I}{[q, m]}} \ldots
 \end{align*}

We define
$\wcG_{(p,n)}:=\bigoplus_{i,j\in \w\I_{(p,n)}} \C E_{i,j}$,
$\cG_{(p,n)}:=\bigoplus_{i,j \in \I_{(p,n)}}\C E_{i,j}$
and $\ovcG[q,m]_{(p,n)}:=\bigoplus_{i,j\in \ov{\I}[q,m]_{(p,n)}} \C E_{i,j}$.
Note that $\cG_{(p,n)} \cong \gl(0|p+n) \cong \gl(p+n)$ and $\ovcG[q,m]_{(p,n)} \cong \gl(p+m|q+n)$. In particular, $\ovcG[0,m]_{(0,n)} \cong \gl(m|n)$.

Let
$\w{\bf b}_{(p,n)}:=\w{\bf b}\cap\wcG_{(p,n)}$, ${\bf b}_{(p,n)}:=\w{\bf b}\cap\cG_{(p,n)}$ and $\ov{{\bf b}}[q,m]_{(p,n)}=\bigoplus_{i,j \in \ov\I[q,m]_{(p,n)}, i \le_{\ov\I[q,m]} j }\C E_{i,j}$
be the standard Borel subalgebras of $\wcG_{(p,n)}$, $\cG_{(p,n)}$ and $\ovcG[q,m]_{(p,n)}$, respectively,
and let $\w{\bf h}_{(p,n)}$, ${\bf h}_{(p,n)}$ and $\ov{\bf h}[q,m]_{(p,n)}$ be the corresponding Cartan subalgebras, which admit bases
$\{E_{i,i}\,|\, i \in \w\I_{(p,n)}\}$,
$\{E_{i,i}  \, | \, i \in \I_{(p,n)}\}$ and
$\{E_{i,i} \,|\,  i\in \ov{\I}[q,m]_{(p,n)}\}$,
respectively.
Let $\{\ep_i\}$ denote the dual bases of the Cartan subalgebras with the corresponding indices.
We will simply write
$$
  E_i:=E_{i,i} \qquad \mbox{for $i \in \hf \Z^*$}.
$$

\subsection{Dynkin diagrams} \label{Dynkin}
\label{dynkin}
Consider the free abelian group with basis $\{\ep_i \, | \, i \in \hf\Zs \}$. It is endowed with a symmetric bilinear form $(\cdot,\cdot)$ defined by
$$
(\ep_i,\ep_j)=(-1)^{2i}\delta_{i,j}, \qquad i,j \in \hf\Zs.
$$
Let
$$
\alpha_i=\ep_i-\ep_{i+\hf} \quad \mbox{ and }\quad \beta_{j}=\ep_j-\ep_{j+1},
$$
where $i,j \in \hf\Zs$, $i \not=-\hf$ and $j \not=-1$.
For $n\in \Zp$, let
\makebox(23,0){$\oval(20,14)$}\makebox(-20,8){$\mc{L}_n$},
\makebox(23,0){$\oval(20,14)$}\makebox(-20,8){$\ov{\mc{L}}_n$},
\makebox(23,0){$\oval(20,14)$}\makebox(-20,8){${\mc{R}}_n$} and
\makebox(23,0){$\oval(20,14)$}\makebox(-20,8){$\ov{\mc{R}}_n$}
denote the following Dynkin diagrams with prescribed fundamental systems:

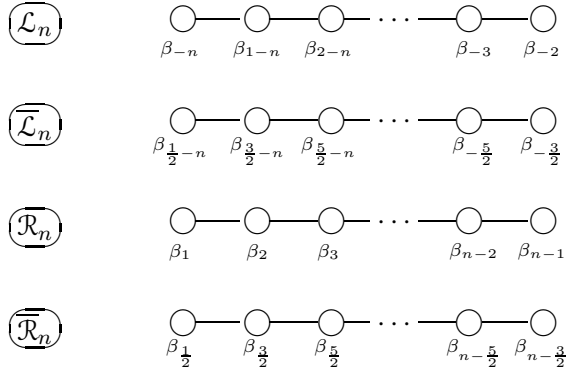
\begin{center}
\hskip -3cm \setlength{\unitlength}{0.16in}
\begin{picture}(22,3)
\put(8,2){\makebox(0,0)[c]{$\bigcirc$}}
\put(10.4,2){\makebox(0,0)[c]{$\bigcirc$}}
\put(14.85,2){\makebox(0,0)[c]{$\bigcirc$}}
\put(17.25,2){\makebox(0,0)[c]{$\bigcirc$}}
\put(5.6,2){\makebox(0,0)[c]{$\bigcirc$}}
\put(8.4,2){\line(1,0){1.55}} \put(10.82,2){\line(1,0){0.8}}
\put(13.2,2){\line(1,0){1.2}} \put(15.28,2){\line(1,0){1.45}}
\put(6,2){\line(1,0){1.4}}
\put(12.5,1.95){\makebox(0,0)[c]{$\cdots$}}
\put(0,1.2){{\ovalBox(1.6,1.3){$\mc{L}_n$}}}
\put(5.5,1){\makebox(0,0)[c]{\tiny$\beta_{-n}$}}
\put(8,1){\makebox(0,0)[c]{\tiny$\beta_{1-n}$}}
\put(10.3,1){\makebox(0,0)[c]{\tiny$\beta_{2-n}$}}
\put(15,1){\makebox(0,0)[c]{\tiny$\beta_{-3}$}}
\put(17.2,1){\makebox(0,0)[c]{\tiny$\beta_{-2}$}}
\end{picture}
\end{center}
\begin{center}
\hskip -3cm \setlength{\unitlength}{0.16in}
\begin{picture}(22,3)
\put(8,2){\makebox(0,0)[c]{$\bigcirc$}}
\put(10.4,2){\makebox(0,0)[c]{$\bigcirc$}}
\put(14.85,2){\makebox(0,0)[c]{$\bigcirc$}}
\put(17.25,2){\makebox(0,0)[c]{$\bigcirc$}}
\put(5.6,2){\makebox(0,0)[c]{$\bigcirc$}}
\put(8.4,2){\line(1,0){1.55}} \put(10.82,2){\line(1,0){0.8}}
\put(13.2,2){\line(1,0){1.2}} \put(15.28,2){\line(1,0){1.45}}
\put(6,2){\line(1,0){1.4}}
\put(12.5,1.95){\makebox(0,0)[c]{$\cdots$}}
\put(0,1.2){{\ovalBox(1.6,1.3){$\ov{\mc{L}}_n$}}}
\put(5.5,1){\makebox(0,0)[c]{\tiny$\beta_{\hf-n}$}}
\put(8,1){\makebox(0,0)[c]{\tiny$\beta_{\frac{3}{2}-n}$}}
\put(10.3,1){\makebox(0,0)[c]{\tiny$\beta_{\frac{5}{2}-n}$}}
\put(15,1){\makebox(0,0)[c]{\tiny$\beta_{-\frac{5}{2}}$}}
\put(17.2,1){\makebox(0,0)[c]{\tiny$\beta_{-\frac{3}{2}}$}}
\end{picture}
\end{center}
\begin{center}
\hskip -3cm \setlength{\unitlength}{0.16in}
\begin{picture}(22,3)
\put(8,2){\makebox(0,0)[c]{$\bigcirc$}}
\put(10.4,2){\makebox(0,0)[c]{$\bigcirc$}}
\put(14.85,2){\makebox(0,0)[c]{$\bigcirc$}}
\put(17.25,2){\makebox(0,0)[c]{$\bigcirc$}}
\put(5.6,2){\makebox(0,0)[c]{$\bigcirc$}}
\put(8.4,2){\line(1,0){1.55}} \put(10.82,2){\line(1,0){0.8}}
\put(13.2,2){\line(1,0){1.2}} \put(15.28,2){\line(1,0){1.45}}
\put(6,2){\line(1,0){1.4}}
\put(12.5,1.95){\makebox(0,0)[c]{$\cdots$}}
\put(0,1.2){{\ovalBox(1.6,1.2){$\mc{R}_n$}}}
\put(5.5,1){\makebox(0,0)[c]{\tiny$\beta_{1}$}}
\put(8,1){\makebox(0,0)[c]{\tiny$\beta_{2}$}}
\put(10.3,1){\makebox(0,0)[c]{\tiny$\beta_{3}$}}
\put(15,1){\makebox(0,0)[c]{\tiny$\beta_{n-2}$}}
\put(17.2,1){\makebox(0,0)[c]{\tiny$\beta_{n-1}$}}
\end{picture}
\end{center}
\begin{center}
\hskip -3cm \setlength{\unitlength}{0.16in}
\begin{picture}(22,3)
\put(8,2){\makebox(0,0)[c]{$\bigcirc$}}
\put(10.4,2){\makebox(0,0)[c]{$\bigcirc$}}
\put(14.85,2){\makebox(0,0)[c]{$\bigcirc$}}
\put(17.25,2){\makebox(0,0)[c]{$\bigcirc$}}
\put(5.6,2){\makebox(0,0)[c]{$\bigcirc$}}
\put(8.4,2){\line(1,0){1.55}} \put(10.82,2){\line(1,0){0.8}}
\put(13.2,2){\line(1,0){1.2}} \put(15.28,2){\line(1,0){1.45}}
\put(6,2){\line(1,0){1.4}}
\put(12.5,1.95){\makebox(0,0)[c]{$\cdots$}}
\put(0,1.2){{\ovalBox(1.6,1.2){$\ov{\mc{R}}_n$}}}
\put(5.5,1){\makebox(0,0)[c]{\tiny$\beta_{\hf}$}}
\put(8,1){\makebox(0,0)[c]{\tiny$\beta_{\frac{3}{2}}$}}
\put(10.3,1){\makebox(0,0)[c]{\tiny$\beta_{\frac{5}{2}}$}}
\put(15,1){\makebox(0,0)[c]{\tiny$\beta_{n-\frac{5}{2}}$}}
\put(17.2,1){\makebox(0,0)[c]{\tiny$\beta_{n-\frac{3}{2}}$}}
\end{picture}
\end{center}
For $n=\infty$, the Dynkin diagrams of
\makebox(23,0){$\oval(20,14)$}\makebox(-20,8){$\mc{L}_n$},
\makebox(23,0){$\oval(20,14)$}\makebox(-20,8){$\ov{\mc{L}}_n$},
\makebox(23,0){$\oval(20,14)$}\makebox(-20,8){${\mc{R}}_n$} and
\makebox(23,0){$\oval(20,14)$}\makebox(-20,8){$\ov{\mc{R}}_n$} are
denoted by \makebox(23,0){$\oval(20,14)$}\makebox(-20,8){$\mc{L}$},
\makebox(23,0){$\oval(20,14)$}\makebox(-20,8){$\ov{\mc{L}}$},
\makebox(23,0){$\oval(20,14)$}\makebox(-20,8){${\mc{R}}$} and
\makebox(23,0){$\oval(20,14)$}\makebox(-20,8){$\ov{\mc{R}}$},
respectively.

The Dynkin diagrams of
$\wcG$, $\cG_{(p,n)}$ and $\ovcG[q,m]_{(p,n)}$
equipped with prescribed fundamental systems are given below.
\bigskip

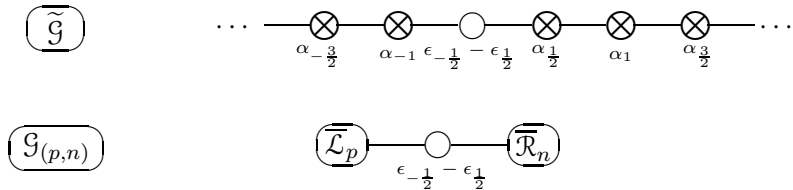
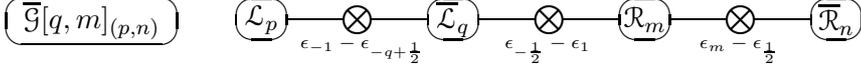
\begin{center}
\hskip 0.5cm \setlength{\unitlength}{0.16in}
\begin{picture}(22,3)

\put(-4,1.2){{\ovalBox(1.8,1.4){$\wcG$}}}

\put(5.6,2){\makebox(0,0)[c]{$\bigotimes$}}
\put(8,2){\makebox(0,0)[c]{$\bigotimes$}}
\put(10.4,2){\makebox(0,0)[c]{$\bigcirc$}}
\put(12.8,2){\makebox(0,0)[c]{$\bigotimes$}}
\put(15.2,2){\makebox(0,0)[c]{$\bigotimes$}}
\put(17.6,2){\makebox(0,0)[c]{$\bigotimes$}}
\put(3.65,2){\line(1,0){1.5}}
\put(6,2){\line(1,0){1.5}}\put(8.4,2){\line(1,0){1.5}}
\put(10.8,2){\line(1,0){1.5}} \put(13.2,2){\line(1,0){1.5}}
\put(15.6,2){\line(1,0){1.5}} \put(18,2){\line(1,0){1.5}}
\put(2.7,1.95){\makebox(0,0)[c]{$\cdots$}}
\put(20.3,1.95){\makebox(0,0)[c]{$\cdots$}}
\put(5.4,1){\makebox(0,0)[c]{\tiny $\alpha_{-\frac{3}{2}}$}}
\put(8,1){\makebox(0,0)[c]{\tiny $\alpha_{-1}$}}
\put(10.3,1){\makebox(0,0)[c]{\tiny $\ep_{-\hf}-\ep_{\hf}$}}
\put(12.8,1){\makebox(0,0)[c]{\tiny $\alpha_{\hf}$}}
\put(15.2,1){\makebox(0,0)[c]{\tiny $\alpha_{1}$}}
\put(17.7,1){\makebox(0,0)[c]{\tiny $\alpha_{\frac{3}{2}}$}}
\end{picture}
\end{center}

\bigskip
\begin{center}
\hskip -2cm \setlength{\unitlength}{0.16in}
\begin{picture}(22,1)

\put(-1.5,-0.3){{\ovalBox(3,1.4){$\cG_{(p,n)}$}}}

\put(9.25,0.5){\makebox(0,0)[c]
{{\ovalBox(1.6,1.3){$\ov{\mc{L}}_{p}$}}}}
\put(10.05,0.5){\line(1,0){1.85}}
\put(12.35,0.5){\makebox(0,0)[c]{$\bigcirc$}}
\put(12.8,0.5){\line(1,0){1.85}}
\put(15.45,0.5){\makebox(0,0)[c]
{{\ovalBox(1.6,1.3){$\ov{\mc{R}}_n$}}}}
\put(12.5,-0.5){\makebox(0,0)[c]
{\tiny$\ep_{-\hf}-\ep_{\hf}$}}
\end{picture}
\end{center}

\bigskip
\begin{center}
\hskip -2cm \setlength{\unitlength}{0.16in}
\begin{picture}(22,1)

\put(-3,-0.3){{\ovalBox(5.5,1.5){$\ovcG[q,m]_{(p,n)}$}}}

\put(5.25,0.5){\makebox(0,0)[c]{{\ovalBox(1.6,1.3){$\mc{L}_p$}}}}
\put(6.05,0.5){\line(1,0){1.85}}
\put(8.35,0.5){\makebox(0,0)[c]{$\bigotimes$}}
\put(8.8,0.5){\line(1,0){1.85}}
\put(11.45,0.5){\makebox(0,0)[c]{{\ovalBox(1.6,1.3){$\ov{\mc{L}}_q$}}}}
\put(12.25,0.5){\line(1,0){1.85}}
\put(14.55,0.5){\makebox(0,0)[c]{$\bigotimes$}}
\put(15.0,0.5){\line(1,0){1.85}}
\put(17.65,0.5){\makebox(0,0)[c]{{\ovalBox(1.6,1.3){${\mc{R}}_m$}}}}
\put(18.45,0.5){\line(1,0){1.85}}
\put(20.75,0.5){\makebox(0,0)[c]{$\bigotimes$}}
\put(21.2,0.5){\line(1,0){1.85}}
\put(23.85,0.5){\makebox(0,0)[c]{{\ovalBox(1.6,1.3){$\ov{\mc{R}}_n$}}}}
\put(8.5,-0.5){\makebox(0,0)[c]{\tiny$\ep_{-1}-\ep_{-q+\hf}$}}
\put(14.5,-0.5){\makebox(0,0)[c]{\tiny$\ep_{-\hf}-\ep_{1}$}}
\put(20.7,-0.5){\makebox(0,0)[c]{\tiny$\ep_{m}-\ep_{\hf}$}}
\end{picture}
\end{center}
\bigskip
Here $\bigotimes$ denotes an odd isotropic simple root.

\subsection{Central extensions}\label{Central Ext}

The general linear Lie superalgebra $\wcG$ has a central extension, denoted by $\DG$, by the one-dimensional center $\C K$  corresponding to the following $2$-cocycle (cf. \cite[p. 99]{CL03}):
$$
\tau(A, B):=\text{Str}([J,A]B),\qquad \mbox{$A,B\in \wcG$,}
$$
where $J:=-\sum_{r \leq -\hf} E_{r}$ and $\text{Str}$ denotes the supertrace. The cocycle $\tau$ is a coboundary, and there is a Lie superalgebra isomorphism $\iota: \wcG\oplus \C K \to \DG$ defined by
\begin{equation}\label{iso-e}
  \iota(A)=A+\mbox{Str}(JA)K,\quad\mbox{for $A \in \wcG$, \quad and \quad $\iota(K)=K$.}
\end{equation}

The central extension
$\G_{(p,n)}$ (resp., $\SG[q,m]_{(p,n)}$) of $\cG_{(p,n)}$ (resp., $\ovcG[q,m]_{(p,n)}$)
is obtained by restricting $\tau$ to $\cG_{(p,n)} \times \cG_{(p,n)}$ (resp., $\ovcG[q,m]_{(p,n)} \times \ovcG[q,m]_{(p,n)}$).
The restriction of the isomorphism $\iota$ to
$\cG_{(p,n)} \oplus \C K$
(resp., $\ovcG[q,m]_{(p,n)} \oplus \C K$)
is an isomorphism
$\iota : \cG_{(p,n)} \oplus \C K \to \G_{(p,n)}$
(resp., $\iota : \ovcG[q,m]_{(p,n)} \oplus \C K \to \SG[q,m]_{(p,n)}$).

The isomorphism \eqnref{iso-e}
allows us to view every
$\DG$(resp., $\G_{(p,n)}$ and $\SG[q,m]_{(p,n)}$)-module as a
$\wcG$(resp., $\cG_{(p,n)}$ and $\ovcG[q,m]_{(p,n)}$)-module.
These central extensions are convenient and conceptual for the formulation of truncation functors and super duality described in Section \ref{SD}; see \cite[Remark 3.3]{CLW11} for further explanations.

\begin{rem} \label{p=q=0}
As $[J, A]=0$ for $A \in \wcG_{(0, \infty)}$, we see that $[A, B]_{\DG_{(0, \infty)}}=[A,B]$ for $A, B\in \wcG_{(0, \infty)}$.
Here $[\cdot, \cdot]_{\DG_{(0, \infty)}}$ denotes the Lie bracket on $\DG_{(0, \infty)}$.
\end{rem}

Let
$\w{\mf b}:=\w{\bf b}\oplus \C K$,
${\mf b}_{(p,n)}:={\bf b}_{(p,n)}\oplus \C K$
and
$\ov{\mf b}[q,m]_{(p,n)}:=\ov{\bf b}[q,m]_{(p,n)} \oplus \C K$
be the standard Borel subalgebras of
$\DG$, $\G_{(p,n)}$ and $\SG[q,m]_{(p,n)}$,
respectively,
and let
$\wfh$,
$\fh_{(p,n)}$ and $\ovfh[q,m]_{(p,n)}$
denote the corresponding Cartan subalgebras of
$\DG$, $\G_{(p,n)}$ and $\SG[q,m]_{(p,n)}$
spanned by bases
$\{K, E_i \}$
with dual bases $\{\Lambda_0,\ep_i \}$ in the restricted dual
$\wfh^{*}$, $\fh_{(p,n)}^*$ and $\ovfh[q,m]_{(p,n)}^*$, where $i$ runs over the index sets
$\hf \Zs$, $\I_{(p,n)}$ and $\ov{\I}[q,m]_{(p,n)}$, respectively.
Here $\La_0$ is the element of $\wfh^{*}$
(resp., $\fh_{(p,n)}^*$ and $\ovfh[q,m]_{(p,n)}^*$) defined by
 $$
 \La_0(K)=1 \,\,\, \textrm{ and }   \,\,\,\La_0(E_i)=0
 $$
  for all $i \in \hf \Zs$ (resp., $\I_{(p,n)}$ and $\ov{\I}[q,m]_{(p,n)}$).

\subsection{Super duality} \label{SD}

We refer the readers to \cite{CLW12} for details on materials in this subsection.

Hereafter for $r \in \Zp$, we denote
$$\langle r  \rangle:=\max\{r, 0\}.$$
Given a partition $\mu=(\mu_1,\mu_2,\ldots)$, we denote by $\ell(\mu)$ the length of $\mu$ and by $\mu'$ the conjugate partition of $\mu$.
We also denote by $\theta(\mu)$ the modified Frobenius coordinates of $\mu$:
\begin{equation*}
\theta(\mu)
:=(\theta(\mu)_{1/2},\theta(\mu)_1,\theta(\mu)_{3/2},\theta(\mu)_2,\ldots),
\end{equation*}
where
$\theta(\mu)_{i-1/2}:=\langle\mu'_i-i+1 \rangle$ and
$\theta(\mu)_i:=\langle \mu_i-i \rangle$ for $i \in \N$.

For any partitions $\la^+=(\la^+_1,\la^+_2,\ldots)$ and $\la^-=(\la^-_1,\la^-_2,\ldots)$, and $d \in \C$, we set
\begin{align}
\w{\la} &:= -\sum_{r \in \hf \N} \theta (\la^-)_r \ep_{-r} +\sum_{i  \in \hf \N}\theta(\la^+)_i \ep_i + d \La_0\in \wfh^{*},\label{weight:wtIm}\\
\la &:= -\sum_{r \in \N}(\la^-)^\prime_r \ep_{-r+\hf} +\sum_{i  \in \N}(\la^+)'_i \ep_{i-\hf}+d \La_0\in \fh^{*},\label{weight:m}\\
\ov\la[q,m] &:=-\sum_{r \in \N} \langle \la^{-}_{r}-q \rangle \ep_{-r}- \sum_{s=1}^{q} (\la^{-})^\prime_{s} \ep_{-s+\hf}  \label{weight:Im} \\  \nonumber
& \hspace{1cm} +\sum_{i=1}^{m} \la^+_{i}\ep_{i}+ \sum_{j \in \N}\langle (\la^+)'_{j}-m \rangle\ep_{j-\hf}+ d\La_0 \in \ovfh[q,m]^*.
\end{align}

Denote by
$\w{\cP}^+(d)  \subseteq \wfh^*$,
${\cP}^+(d)  \subseteq \fh^*$
and $\ov{\cP}[q,m]^+(d)  \subseteq \ovfh[q,m]^*$
the sets of all weights of the forms \eqref{weight:wtIm}, \eqref{weight:m} and \eqref{weight:Im}, respectively.
Let $\w{\cP}^+=\cup_{d \in \C} \w{\cP}^+(d)$,
${\cP}^+=\cup_{d \in \C} {\cP}^+(d)$
and $\ov{\cP}[q,m]^+=\cup_{d \in \C} \ov{\cP}[q,m]^+(d)$.
By definition, we have the bijections
\begin{equation}\label{cP}
\begin{array}{ll}
\w{\cP}^+ \longrightarrow {\cP}^+ & \raisebox{-4pt}{and}\ \ \ \ \ \
\w{\cP}^+ \longrightarrow \ov{\cP}[q,m]^+\\
 \ \     \w\la\ \mapsto \ \la &
 \ \   \ \ \ \ \ \ \ \ \  \ \  \w\la\ \mapsto \ \ovla[q,m].
\end{array}
\end{equation}
If $\theta (\la^+)_{n+\hf}=\theta (\la^-)_{p+\hf}=0$ (resp., $(\la^+)'_{n+1}=(\la^-)'_{p+1}=0$, and $(\la^+)'_{n+1}\leq m$ and $\la^{-}_{p+1}\leq q$),
then we may view $\wla  \in \w \fh_{(p,n)}^*$ (resp., $\la  \in \fh_{(p,n)}^*$ and $\ovla[q, m] \in \ovfh[m]_{(p,n)}^*$) in a natural way. The set of all such weights is denoted by $\w \cP^+_{(p,n)}$ (resp., $\cP^+_{(p,n)}$ and $\ov{\cP}[q,m]^+_{(p,n)}$).

We denote by
$\w{\Phi}^+$ (resp., ${\Phi}_{(p,n)}^+$ and $\ov{\Phi}[q,m]_{(p,n)}^+$)
the set of all positive roots of
$\DG$ (resp., $\G_{(p,n)}$ and $\SG[q,m]_{(p,n)}$).
Let
\begin{align}
\w{Y}&=\setc*{\ep_i-\ep_j }{i,j\in \hf \Zs,\, i<j \le -\hf \ \  \mbox{or}\ \  \hf \le i<j} \subseteq \w{\Phi}^+, \nonumber \\
{Y}_{(p,n)}&=\setc*{\ep_i-\ep_j }{i,j\in \I_{(p,n)},\, i<j \le -\hf \ \  \mbox{or}\ \  \hf \le i<j } \subseteq {\Phi}_{(p,n)}^+, \label{Levi}\\
\ov{Y}[q,m]_{(p,n)}&=\setc*{\ep_i-\ep_j }{i,j\in \ov\I[q,m]_{(p,n)},\, i<_{\ov{\I}{[q, m]}}j \le_{\ov{\I}{[q, m]}} -\hf \ \  \mbox{or}\ \  \hf \le_{\ov{\I}{[q, m]}} i<_{\ov{\I}{[q, m]}} j }\nonumber\\ &\subseteq \ov{\Phi}[q,m]_{(p,n)}^+.\nonumber
\end{align}
Let
$$\w{\mf l}=\w{\mf h}\oplus  \bigoplus\limits_{\pm\alpha\in\w{Y}} \DG_{\alpha},
 \ \ \ \ {\mf{l}}_{(p,n)}={\mf h}_{(p,n)}\oplus \bigoplus\limits_{\pm\alpha\in{Y}_{(p,n)}} \left( \G_{(p,n)} \right)_{\! \alpha}$$
and
$$\ov{\mf{l}}[q,m]_{(p,n)}= \ovfh[q,m]_{(p,n)} \oplus \bigoplus\limits_{\pm\alpha\in\ov{Y}[q,m]_{(p,n)}} \left(\SG[q,m]_{(p,n)} \right)_{\! \alpha}$$
be the Levi subalgebras of
$\DG$, $\G_{(p,n)}$ and $\SG[q,m]_{(p,n)}$
associated to $\w{Y}$, $Y_{(p,n)}$ and $\ov{Y}[q,m]_{(p,n)}$, respectively,
and let
$\w{\mf p}=\tilde{\mf l}+\w{\mf b}$,
${\mf{p}}_{(p,n)}={\mf{l}}_{(p,n)}+{\mf b}_{(p,n)}$
and $\ov{\mf{p}}[q,m]_{(p,n)}=\ov{\mf{l}}[q,m]_{(p,n)}+\ov{\mf b}[q,m]_{(p,n)}$ be the corresponding parabolic subalgebras.

For $\mu\in \w{\mf h}^*$, let $L(\tilde{\mf l}, \mu)$ be the irreducible highest weight $\tilde{\mf l}$-module with highest weight $\mu$.
We denote by ${\Delta}(\DG,\mu):=\mbox{Ind}^{\DG}_{\w{\mf p}} L(\tilde{\mf l}, \mu)$
the parabolic Verma $\DG$-module
and by
$L(\DG, \mu)$
the unique irreducible quotient
$\DG$-module of
$\Delta(\DG, \mu)$.
For $\mu \in \fh_{(p,n)}^*$
(resp., $\ovfh[q,m]_{(p,n)}^*$),
the modules
$L({\mf l}_{(p,n)}, \mu)$
(resp., $L(\ov{\mf l}[q,m]_{(p,n)}, \mu)$)
and
$\Delta(\G_{(p,n)}, \mu)$
(resp., $\Delta(\SG[q,m]_{(p,n)}, \mu)$)
are defined in a similar way.
We denote by
$L(\G_{(p,n)}, \mu)$
(resp., $L(\SG[q,m]_{(p,n)}, \mu)$)
the unique irreducible quotient
$\G_{(p,n)}$(resp., $\SG[q,m]_{(p,n)}$)-module of $\Delta(\G_{(p,n)}, \mu)$
(resp., $\Delta(\SG[q,m]_{(p,n)}, \mu)$).

Analogous to \cite{CL10, CLW11,CLW12},
we let
$\w\OO$
(resp., $\OO_{(p,n)}$ and $\ov\OO[q,m]_{(p,n)}$)
be the category of
$\DG$(resp., $\G_{(p,n)}$ and $\SG[q,m]_{(p,n)}$)-modules $M$ such that $M$ is a semisimple
$\w{\mf h}$(resp., ${\mf h}_{(p,n)}$ and $\ov{\mf h}[q,m]_{(p,n)}$)-module with finite-dimensional weight subspaces $M_{\gamma}$,
where $\gamma \in \w{\mf h}^*$
(resp., ${\mf h}_{(p,n)}^*$ and $\ov{\mf h}[q,m]_{(p,n)}^*$), satisfying the conditions:
\begin{enumerate}[\normalfont(i)]

\item $M$ decomposes over $\tilde{\mf l}$
(resp., ${\mf l}_{(p,n)}$ and
 $\ov{\mf l}[q,m]_{(p,n)}$)
 as a direct sum of
$L(\tilde{\mf l}, \mu)$
(resp., $L({\mf l}_{(p,n)}, \mu)$ and
$L(\ov{\mf l}[q,m]_{(p,n)}, \mu)$) for $\mu\in \w{\cP}^+$
(resp., ${\cP}^+_{(p,n)}$ and $\ov{\cP}[q,m]^+_{(p,n)}$).

\item There exist finitely many weights $\la^1,\la^2,\ldots,\la^k\in \w{\cP}^+$
(resp., ${\cP}^+_{(p,n)}$ and $\ov{\cP}[q,m]^+_{(p,n)}$) (depending on $M$) such that if $\gamma$ is a weight of $M$, then
$\la^i-\gamma$ is a linear combination of simple roots with coefficients in $\Zp$ for some $i$.

   \end{enumerate}
The morphisms in the categories are even homomorphisms of modules, and the categories are abelian. Also, there is a natural $\Z_2$-gradation on each module in the categories with a compatible action of the corresponding Lie (super)algebra to be defined below.

Let us set
 \begin{align} \nonumber
\w{\Xi}_{(p,n)}&=-\sum_{r \in  \w\I_{(p,n)} \cap \hf \N} \Zp\ep_{-r}+\sum_{i \in \w\I_{(p,n)} \cap\hf \N} \Zp\ep_{i}+\C\Lambda_0,  \\ \label{weight}
{\Xi}_{(p,n)}&=-\sum_{r \in{\I}_{(p,n)}\cap \hf \N} \Zp\ep_{-r}+\sum_{i \in{\I}_{(p,n)}\cap \hf \N} \Zp\ep_{i}+\C\Lambda_0, \\ \nonumber
  \ov{\Xi}[q,m]_{(p,n)}&=-\sum_{r \in  \ov{\I}[q,m]_{(p,n)} \cap \hf \N} \Zp\ep_{-r}+\sum_{i \in \ov{\I}[q,m]_{(p,n)} \cap\hf \N} \Zp\ep_{i}+\C\Lambda_0.
 \end{align}
For $\varepsilon=0$ or $1$, and $\Theta=\w{\Xi}$, ${\Xi}_{(p,n)}$ or ${\ov\Xi}[q,m]_{(p,n)}$, we define
$$
{\Theta}(\ov\varepsilon):=
\setc[\Big]{\mu\in {\Theta} }{ \sum_{i \in \Z}\mu(E_{i+\hf})\equiv \varepsilon \,\,(\text{mod }2)}.
$$
For $M\in \w\OO$
(resp., ${\OO}_{(p,n)}$ and
$\ov\OO[q,m]_{(p,n)}$),
all the weights of $M$ lie in $\w{\Xi}$ (resp., ${\Xi}_{(p,n)}$ and $\ov{\Xi}[q,m]_{(p,n)}$)
(see the paragraph preceding \cite[Theorem 6.4]{CW}). For $M\in \w\OO$, ${M}={M}_{\ov{0}}\bigoplus {M}_{\ov{1}}$  is a $\Z_2$-graded vector space, where
 \begin{equation}\label{wt-Z2-gradation}
{M}_{\ov{0}}:=\bigoplus_{\mu \in \w{\Xi}(\ov 0)}{M}_{\mu}\qquad\mbox{and}\qquad
{M}_{\ov{1}}:=\bigoplus_{\mu \in \w{\Xi}(\ov 1)}{M}_{\mu}.
 \end{equation}
The $\Z_2$-gradation on $M$ is clearly compatible with the action of $\DG$.
Similarly, we may define a $\Z_2$-gradation with a compatible action of ${\mf g}_{(p,n)}$ (resp., $\ov{\mf g}[q,m]_{(p,n)}$) on ${M}$ for $M \in {\OO}_{(p,n)}$ (resp., $\ov\OO[q,m]_{(p,n)}$).
By \cite[Theorem 3.27 and Theorem 6.4]{CW} (see also the proof of \cite[Theorem 6.2.2]{Lus}), $\w\OO$, $\OO_{(p,n)}$ and $\ov\OO[q,m]_{(p,n)}$ are tensor categories.
Note that the $\Z_2$-gradation on $M\otimes N$ given by \eqnref{wt-Z2-gradation} and the $\Z_2$-gradation on $M\otimes N$ induced from the $\Z_2$-gradations on $M$ and $N$ given by
\eqref{wt-Z2-gradation} are the same for $M, N\in \w\OO$ (resp., $\OO_{(p,n)}$ and $\ov\OO[q,m]_{(p,n)}$).
Our discussion can be summarized as follows.

\begin{prop} \label{weight-Xi}
\begin{enumerate} [\normalfont(i)]

\item The weights of modules in $\w\OO$ (resp., $\OO_{(p,n)}$ and $\ov\OO[q,m]_{(p,n)}$) are contained in $\w{\Xi}$ (resp., ${\Xi}_{(p,n)}$ and $\ov{\Xi}[q,m]_{(p,n)}$).

\item $\w\OO$, $\OO_{(p,n)}$ and $\ov\OO[q,m]_{(p,n)}$ are tensor categories.

\end{enumerate}
\end{prop}

We also have the following proposition (cf. \cite[Lemma 4.1]{CLW12}).

\begin{prop} \label{tensor-cat}
\begin{enumerate}[\normalfont(i)]
  \item A highest weight $\DG$-module with highest weight $\la$ lies in $\w\OO$ if and only if  $\la \in \w \cP^+$.
	
	\item A highest weight $\G_{(p,n)}$-module with highest weight $\la$ lies in $\OO_{(p,n)}$  if and only if  $\la \in \cP_{(p,n)}^+$.
	
 \item A highest weight $\SG[q,m]_{(p,n)}$-module with highest weight $\la$ lies in $\ov\OO[q,m]_{(p,n)}$ if and only if  $\la \in \ov \cP[q,m]_{(p,n)}^+$.
		
\end{enumerate}
\end{prop}

The following lemma can be obtained by the description of the weights of modules in \eqref{weight}.
\begin{lem} \label{weight-decomposition}
Let $M, N\in \w\OO$
(resp., $\OO_{(p,n)}$ and
$\ov\OO[q,m]_{(p,n)}$).
Suppose that $\mu$ and $\gamma$ are weights of $M$ and $N$, respectively. Then
\[
(\mu+\gamma)(E_i)=0\quad\mbox{if and only if }\quad \mu(E_i)=0\,\, \mbox{and}\,\, \gamma(E_i)=0,
 \]
for $i \in \hf\Zs$ (resp., ${\I}_{(p,n)}$ and $\ov{\I}[q,m]_{(p,n)}$).
\end{lem}

Similarly, in view of \eqref{weight}, we immediately obtain the following.

\begin{lem} \label{weight-decomposition2}q
Let $\mu, \gamma\in \w\Xi$. Then:
\begin{enumerate}[\normalfont(i)]
  \item $\mu+\gamma \in  {\Xi}$ if and only if $\mu \in {\Xi}$ and $\gamma \in {\Xi}$.
      \item $\mu+\gamma \in  \ov{\Xi}[q,m]$ if and only if $\mu \in \ov{\Xi}[q,m]$ and $\gamma \in \ov{\Xi}[q,m]$.
\end{enumerate}
\end{lem}

 Let $0 \le r \le p \leq\infty$ and $0 < k \le n \leq\infty$. The truncation functor
${\mf{tr}}^{(p,n)}_{(r,k)} : \OO_{(p,n)} \to \OO_{(r,k)}$ is defined by
$$
{\mf{tr}}^{(p,n)}_{(r,k)} (M)=\bigoplus_{\nu\in {\Xi}_{(r,k)}}M_{\nu} \qquad \mbox{for $M \in \OO_{(p,n)}$}.
$$
The direct sum $\bigoplus_{\nu\in {\Xi}_{(r,k)}}M_{\nu}$ lies in $\OO_{(r,k)}$ by \cite[Lemma 6.1]{CLW12} together with the fact that $[E_i, A]=0$ for all $A\in\G_{(r,k)}$ and all $i \in {\I}_{(p,n)}$ with $i< -r$ or $i>k$. For  every $f \in {\rm Hom}_{\OO_{(p,n)}}(M, N)$, ${\mf{tr}}^{(p,n)}_{(r,k)}(f)$ is defined to be the restriction of $f$ to ${\mf{tr}}^{(p,n)}_{(r,k)}(M)$.
Similarly, we also have the truncation functor
 $\ov{\mf{tr}}[q,m]^{(p,n)}_{(r,k)} : \ov\OO[q,m]_{(p,n)}\to \ov\OO[q,m]_{(r,k)}$ defined by
$$
\ov{\mf{tr}}[q,m]^{(p,n)}_{(r,k)} (M)=\bigoplus_{\nu\in \ov{\Xi}[q,m]_{(r,k)}}M_{\nu} \qquad \mbox{for $M \in \ov\OO[q,m]_{(p,n)}$},
$$
and for every $f \in {\rm Hom}_{\OO[q,m]_{(p,n)}}(M, N)$, $\ov{\mf{tr}}[q,m]^{(p,n)}_{(r,k)}(f)$ is defined to be the restriction of $f$ to $\ov{\mf{tr}}[q,m]^{(p,n)}_{(r,k)}(M)$.
It is clear that ${\mf{tr}}^{(p,n)}_{(r,k)}$ and
$\ov{\mf{tr}}[q,m]^{(p,n)}_{(r,k)}$
are exact functors.
By \lemref{weight-decomposition}, we immediately have the following lemma.

\begin{lem} \label{trun-tensor}
For $0 \le r \le p \leq\infty$ and $0 < k \le n \leq\infty$,
${\mf{tr}}^{(p,n)}_{(r,k)}$ and
$\ov{\mf{tr}}[q,m]^{(p,n)}_{(r,k)}$
are tensor functors.
\end{lem}

The proof of the following result is similar to that of \cite[Lemma 3.2]{CLW11}.
\begin{prop} \label{trun}
 Let $0 \le r \le p \leq\infty$ and $0 < k \le n \leq\infty$. Let $V$ denote either $\Delta$ or $L$.

 \begin{enumerate}[\normalfont(i)]
   \item For ${\mu}\in {\cP}_{(p,n)}^+$,
 $$
 {{\mf{tr}}}^{(p,n)}_{(r,k)} \left(V(\G_{(p,n)},\mu)\right)=
 \begin{cases}
       V(\G_{(r,k)},\mu) & \mbox{if}\ \ {\mu}\in  {\cP}_{(r,k)}^+, \\
       0 & \mbox{otherwise}.
     \end{cases}
 $$
   \item For ${\mu}\in \ov{\cP}[q,m]_{(p,n)}^+$,
 $$
 \ov{\mf{tr}}[q,m]^{(p,n)}_{(r,k)} \left(V(\SG[q,m]_{(p,n)}, \mu)\right)=
 \begin{cases}
       V(\SG[q,m]_{(r,k)}, \mu)& \mbox{if}\ \ {\mu}\in  \ov{\cP}[q,m]_{(r,k)}^+, \\
       0 & \mbox{otherwise}.
     \end{cases}
 $$
 \end{enumerate}
\end{prop}

Given ${M}=\bigoplus_{\gamma \in \wfh^*}{M}_\gamma \in \w\OO$,  we define the functors
$$
{T}({M})
=\bigoplus_{\gamma\in{\fh^*}}{M}_\gamma
\ \ \ \ \ \ \mbox{and}\ \ \ \  \
 \ov{T}_{[q,m]}({M})
 =\bigoplus_{\gamma\in{\ovfh[q,m]^*}}{M}_\gamma.
$$
For $M, N \in \w\OO$
and $f\in {\rm Hom}_{\w\OO}(M,N)$, ${T}({f})$
(resp., $\ov{T}_{[q,m]}({f})$)
is defined to be the restriction of $f$ to
${T}({M})$
(resp., $\ov{T}_{[q,m]}({M})$).
Note that
$f({T}({M})) \subseteq {T}({N})$
(resp., $f(\ov{T}_{[q,m]}({M})) \subseteq \ov{T}_{[q,m]}({N})$),
and
${T}({f}) :  {T}({M}) \to  {T}({N})$
(resp., $\ov{T}_{[q,m]}({f}) :  \ov{T}_{[q,m]}({M}) \to  \ov{T}_{[q,m]}({N})$)
is a $\G$(resp., $\SG[q,m]$)-homomorphism.

\begin{prop} \label{exact-functor}
For $q, m \in \Zp$, the functors $T$ and $\ov{T}_{[q,m]}$ are exact.
\end{prop}

The above proposition can be proved in the same way as in \cite[Proposition 4.4]{CLW12}.
Thanks to \lemref{weight-decomposition2},
${T}(M\otimes N)= {T}(M)\otimes {T}(N)$
(resp., $\ov{T}_{[q,m]}(M\otimes N)= \ov{T}_{[q,m]}(M)\otimes \ov{T}_{[q,m]}(N)$)
 for all $M, N\in \w\OO$, and we obtain the following proposition.

\begin{prop} \label{tensor-functor}
For $q, m \in \Zp$, the functors $T$ and $\ov{T}_{[q,m]}$ are tensor functors.
 \end{prop}

Super duality asserts that the categories
$\w{\mc{O}}$, $\OO$ and $\OO[q,m]$
 are equivalent. It can be described more precisely as in the following theorem.

\begin{thm} \label{thm:SD}
For $q, m \in \Zp$, the functor
${T}: \w\OO\rightarrow\OO$
(resp., $\ov T_{[q,m]}: \w\OO\rightarrow\ov\OO[q,m]$) is an equivalence of tensor categories.
Moreover,
$T$ (resp., $\ov{T}_{[q,m]}$)
 sends parabolic Verma modules to parabolic Verma modules and irreducible modules to irreducible modules. More precisely, for $\wla \in \w{\cP}^+$, we have
 $$
{T}\big{(} {\Delta}(\DG,\wla) {)}={\Delta}(\G,{\la}), \quad {T} \big{(} L(\DG,\wla) \big{)}=L(\G, {\la})
$$
$$
\mbox{(resp., $\ov{T}_{[q,m]}\big{(} {\Delta}(\DG,\wla) {)}={\Delta}(\SG[q,m],\ov{\la}[q,m]), \quad \ov T_{[q,m]}\big{(} L(\DG,\wla) \big{)}=L(\SG[q,m], \ov{\la}[q,m])$)}.
$$
\end{thm}

The theorem can be proved by an argument similar to the proof of the super duality in \cite{CL10,CLW11} (see \cite[Theorem 4.13]{CLW12} as well).

\begin{rem} \label{rem:SD}
For $p=q=0$, we may define the functors $T_\ze: \w\OO_{(0, \infty)} \rightarrow \OO_{(0, \infty)}$ and $\ov T_{[m]}: \w\OO_{(0, \infty)} \rightarrow \ov\OO[0,m]_{(0, \infty)}$ in the same way as we define the functors $T$ and $\ov T_{[q,m]}$. \propref{exact-functor}, \propref{tensor-functor} and \thmref{thm:SD} are also valid for $T_\ze$ and $\ov T_{[m]}$ (cf. \cite[Theorem 2.11]{ChL}).
\end{rem}

\subsection{Unitarizable modules over $\ovcG[q,m]_{(p,n)}$}\label{sec:UM}

 We begin by reviewing the basic notions of $*$-superalgebras and unitarizable modules. A \emph{$*$-superalgebra} is an associative superalgebra $A$ together with an anti-linear anti-involution $\omega: A\to A$ of degree $\ov 0$. A homomorphism $f : (A,\omega) \to (A^\prime, \omega^\prime)$ of $*$-superalgebras is a homomorphism of superalgebras satisfying $\omega^\prime \circ f = f \circ \omega$. Let $(A,\omega)$ be a $*$-superalgebra and $V$ a $\Z_2$-graded $A$-module.  A Hermitian form $\langle\cdot|\cdot\rangle$ on $V$ is said to be \emph{contravariant} if $\langle av | v'\rangle=\langle v|\omega(a)v'\rangle$ for all $a\in A$ and $v,v'\in V$. An $A$-module equipped with a positive definite contravariant Hermitian form is called a \emph{unitarizable} $A$-module. It is evident that any unitarizable $A$-module is completely reducible.

 A Lie superalgebra $\mf L$ is said to admit a \emph{$*$-structure} if $\mf L$ is equipped with an anti-linear anti-involution $\omega$ of degree $\ov 0$. In this case, $\omega$ is also called a $*$-structure on $\mf L$. A homomorphism $f : (\mf L, \omega) \to (\mf L^\prime, \omega^\prime)$ of Lie superalgebras with $*$-structures is a homomorphism of Lie superalgebras satisfying $\omega^\prime \circ f =f \circ \omega$. Moreover, it is clear that $\omega$ is a $*$-structure on $\mf L$ if and only if the natural extension of $\omega$ to the universal enveloping algebra $U(\mf L)$ of $\mf L$ is an anti-linear anti-involution. Let $(\mf L, \omega)$ be a Lie superalgebra with $*$-structure and $V$ a $\Z_2$-graded $\mf L$-module. A Hermitian form $\langle\cdot|\cdot\rangle$ on $V$ is said to be \emph{contravariant} if $\langle xv|v'\rangle=\langle v|\omega(x)v'\rangle$ for all $x\in \mf L$ and $v,v'\in V$. An $\mf L$-module equipped with a positive definite contravariant Hermitian form is called a \emph{unitarizable} $\mf L$-module. Note that an $\mf L$-module $V$ is a unitarizable $\mf L$-module if and only if $V$ is a unitarizable $U(\mf L)$-module. We refer the reader to \cite{ChL, CLZ, LZ06, LZ11} for further details.

The Lie superalgebra $\DG$ admits a $*$-structure $\omega$ defined by
$$
\sum_{i, j \in \hf \Zs} a_{ij} E_{i, j} \mapsto \sum_{i, j \in \hf \Zs} (-1)^{\tau_i+\tau_j} \ov{a}_{ij} E_{j, i} \qquad\mbox{and}\qquad K \mapsto K.
$$
Here $\ov{a}_{ij}$ denotes the complex conjugate of $a_{ij} \in \C$ and
$$
\tau_i:=
 \begin{cases}
       1 & \mbox{if}  \ \  -i \in \N;\\
       0 & \mbox{if}  \ \    -i  \in \hf \Zs \backslash \N.
     \end{cases}
 $$
The above $*$-structure is, in fact, a restriction of the $*$-structure $\omega'$ defined in \cite[Section 3.1]{ChL}.
Evidently, the restriction of $\omega$ to the subalgebra $\SG[q,m]_{(p,n)}$ gives a $*$-structure on $\SG[q,m]_{(p,n)}$, which is also denoted by $\omega$.
Moreover, the restriction of $\omega$ on $\SG[q,m]_{(p,n)}$ to $\ovcG[q,m]_{(p,n)}$, via the isomorphism $\iota$ in \eqref{iso-e}, gives a $*$-structure on $\ovcG[q,m]_{(p,n)}$, which we denote again by $\omega$ (see also \cite[p. 791]{CLZ}).

Let $d \in \N$. A \emph{generalized partition} $\la:=(\la_1,\ldots,\la_d)$ of depth $d$ is defined to be a sequence of integers in decreasing order: $\la_1 \ge \ldots \ge \la_d$. We denote by $\cP_d$ the set of all generalized partitions of depth $d$.

Suppose $\la \in \cP_d$. Define
$$
\la^+:=\left(\langle \la_1 \rangle,\ldots,\langle \la_d \rangle \right) \quad \mbox{and} \quad
\la^-:=\left(\langle -\la_d \rangle,\ldots,\langle -\la_1 \rangle \right).
$$
Then $\la^+$ and $\la^-$ are partitions. Let
$\cP_d(p+m|q+n)$
be the set of all generalized partitions $\la$ of depth $d$ such that $\la_{m+1}\le n$ and $\la_{d-p}\ge -q$.
The condition $\la_{d-p}\ge -q$ is considered to be automatically satisfied whenever $d \le p$.

For any $\la \in \cP_d(p+m|q+n)$, we define $\ovla \in \ov{\bf h}[q,m]_{(p,n)}^*$ by
$$
\ovla:=-\sum_{r=1}^{p} \langle \la^{-}_{r}-q \rangle \ep_{-r}- \sum_{s=1}^{q} (\la^{-})^\prime_{s} \ep_{-s+\hf}  +\sum_{i=1}^m \la^+_i \, \ep_i+\sum_{j=1}^n\langle(\la^+)'_j-m\rangle \,\ep_{j-\hf}.
$$
Note that $(\la^+)'_{n+1}\leq m$, $\la^{-}_{p+1}\leq q$ and $\ell(\la^+)+\ell(\la^-)\le d$.
We denote
$$
\mathbbm{1}_{p|q}=\sum_{r=1}^{p} \ep_{-r}-\sum_{s=1}^{q} \ep_{-s+\hf}.
$$
Let
\begin{equation}\label{un-weight}
\ov\cQ[q,m]_{(p,n)}=\setc*{\ovla - d \mathbbm{1}_{p|q}}{  \la \in \cP_d(p+m|q+n), d \in \N}.
\end{equation}
We denote by
$$L(\ovcG[q,m]_{(p,n)}, \xi) \quad \mbox{(resp., $L(\cG_{(p,n)}, \xi)$)}$$
the irreducible highest weight $\ovcG[q,m]_{(p,n)}$(resp., $\cG_{(p,n)}$)-module with highest weight $\xi \in \ov{\bf h}[q,m]_{(p,n)}^*$ (resp., ${\bf h}_{(p,n)}^*$).
We have the following proposition (cf. \cite{CLZ}).

\begin{prop} \label{CLZ}
Let $m,p,q \in \Zp$, $n \in \N$, and $\xi \in \ov{\bf h}[q,m]_{(p,n)}^*$ an integral weight. Then $L(\ovcG[q,m]_{(p,n)}, \xi)$ is unitarizable with respect to $\omega$ if and only if $\xi \in \ov\cQ[q,m]_{(p,n)}$.
\end{prop}

\section{Quadratic Gaudin Hamiltonians and super KZ equations for $\DG$, $\G_{(p,n)}$ and $\SG[q,m]_{(p,n)}$} \label{sec:quad-G}

In this section, we study the (super) quadratic Gaudin Hamiltonians associated to the Lie (super)algebras $\DG$, $\G_{(p,n)}$ and $\SG[q,m]_{(p,n)}$. We establish linear isomorphisms relating the (generalized) eigenspaces of each quadratic Gaudin Hamiltonian for $\DG$ to the (generalized) eigenspaces of the corresponding quadratic Gaudin Hamiltonian for $\G$ (resp., $\SG[q,m]$).

On the other hand, we introduce the (super) Knizhnik-Zamolodchikov (KZ) equations for the Lie (super)algebras $\SG[q,m]_{(p,n)}$ and $\G_{(p,n)}$. We see that the solutions of these equations are stable under truncation functors, and that the singular solutions of the super KZ equations for $\SG[q,m]$ and $\G$ are related.

\subsection{Quadratic Gaudin Hamiltonians on modules over $\DG$, $\G_{(p,n)}$ and $\SG[q,m]_{(p,n)}$} \label{sec:QGH}

For $i \in \hf \Zs$, let
\begin{equation*}
\delta_i:=
 \begin{cases}
       1 & \mbox{if}  \ \    i<0;\\
       0 & \mbox{if}  \ \   i>0.
     \end{cases}
\end{equation*}
Define the \emph{Casimir symmetric tensors} ${\w{\Omega}}$, ${\Omega}_{(p,n)}$ and ${\ov{\Omega}}[q,m]_{(p,n)}$ for $\DG$, $\G_{(p,n)}$ and $\SG[q,m]_{(p,n)}$ respectively by
\begin{align}
{\w{\Omega}}&:= \sum_{i,j \in \hf \Zs} (-1)^{2j} E_{i,j} \otimes E_{j,i}-\sum_{i \in \hf \Zs} \delta_i (K\otimes E_i+E_i \otimes K), \nonumber \\
{\Omega}_{(p,n)}&:= \sum_{i,j \in \I_{(p,n)}} (-1)^{2j} E_{i,j} \otimes E_{j,i}-\sum_{i \in \I_{(p,n)}} \delta_i (K\otimes E_i+E_i \otimes K), \label{Casimir} \\
{\ov{\Omega}}[q,m]_{(p,n)}&:= \sum_{i,j \in \ov{\I}{[q, m]}_{(p, n)}} (-1)^{2j} E_{i,j} \otimes E_{j,i}-\sum_{i \in \ov{\I}{[q, m]}_{(p, n)}} \delta_i (K\otimes E_i+E_i \otimes K).\nonumber
\end{align}
Here we drop the subscript $(p,n)$ for $(p,n)=(\infty, \infty)$.

For $p \in \Zp$ and $n \in \N$, the Casimir symmetric tensor $\Omega_{(p,n)}$ (resp., $\ov{\Omega}[q,m]_{(p,n)}$) lies in $U(\G_{(p,n)}) \otimes U(\G_{(p,n)})$ (resp., $U(\SG[q,m]_{(p,n)}) \otimes U(\SG[q,m]_{(p,n)}))$ and acts on $M_1 \otimes M_2$ for $M_1, M_2 \in \OO_{(p,n)}$ (resp., $\ov\OO[q,m]_{(p,n)}$).

 \begin{lem} \label{lem:finite_sum}
Let $0 \le r \le p \leq\infty$ and $0 < k \le n \leq\infty$.
Suppose that $v$ is a weight vector of weight $\mu$ in $M  \in \w\OO$ (resp., $\OO_{(p,n)}$ and $\ov\OO[q,m]_{(p,n)}$).
If $\mu \in \w\Xi_{(r,k)}$ (resp., $\Xi_{(r,k)}$ and $\ov\Xi[q,m]_{(r,k)}$), then $E_i v= 0$ for all $i \in \hf \Zs \backslash \w{\I}_{(r,k)}$ (resp., $\I_{(p,n)} \backslash \I_{(r,k)}$ and
$\ov{\I}{[q, m]}_{(p,n)}\backslash \ov{\I}{[q, m]}_{(r,k)}$),
and $E_{i,j} v= 0$ for all $i,j \in \hf \Zs$ (resp., $\I_{(p,n)}$ and $\ov{\I}{[q, m]}_{(p,n)}$) such that $i < j$ (resp., $i < j$ and $i <_{\ov{\I}{[q, m]}} j$) and at least one of $i$ and $j$ does not lie in $\w{\I}_{(r,k)}$ (resp., $\I_{(r,k)}$ and $\ov{\I}{[q, m]}_{(r,k)}$).
 \end{lem}

\begin{proof}
 We will only prove the case where
 $M \in \ov\OO[q,m]_{(p,n)}$.
 The other cases can be proved in the same way.
 Let $M \in \ov\OO[q,m]_{(p,n)}$,
 and let $v \in M$ be a weight vector of weight $\mu$. Suppose $\mu \in \ov\Xi[q,m]_{(r,k)}$.
 For any
 $i \in \ov{\I}{[q, m]}_{(p,n)}\backslash \ov{\I}{[q, m]}_{(r,k)}$,
 either $i<_{\ov{\I}{[q, m]}} -r$ or $i>_{\ov{\I}{[q, m]}}k$.
 Clearly, $\mu(E_i)=0$ and hence $E_i v=0$.
Now let $i, j \in \ov{\I}{[q, m]}_{(p,n)}$ with $i <_{\ov{\I}{[q, m]}} j$. Suppose that $i \notin \ov{\I}{[q, m]}_{(r,k)}$. Assume $E_{i,j}v \not=0$, and let $\gamma$ be the weight of $E_{i,j}v$. Then for $i<0$, we have $\gamma(E_i)=1$, and for $i>0$, we have $j>0$ and $\gamma(E_j)=-1$. It follows that $\gamma \notin  \ov\Xi[q,m]_{(p,n)}$, which is a contradiction. Consequently, $E_{i,j}v=0$. The same reasoning also shows that $E_{i,j}v=0$ if $j \notin \ov{\I}{[q, m]}_{(r,k)}$. This completes the proof.
\end{proof}
	
For any $M \in \w\OO$ (resp., $\OO$ and $\ov\OO[q,m]$) and $v\in M$, there are only finitely many $E_i$ and $E_{i,j}$ such that $E_i v\not= 0$ and $E_{i,j} v\not= 0$ for $i < j$ (resp., $i < j$ and $i <_{\ov{\I}{[q, m]}} j$) by \lemref{lem:finite_sum}. Consequently, the operator
$\w{\Omega}$ (resp., ${\Omega}$ and $\ov{\Omega}[q,m]$)
is well defined on $M_1 \otimes M_2$ for $M_1, M_2 \in \w\OO$ (resp., $\OO$ and $\ov\OO[q,m]$).

{\bf  Suppose $\ell \in \N$ with $\ell \ge 2$.
From now on, we use the symbol ${M}^{\otimes}\in \w\OO$ (resp., $\OO_{(p,n)}$ and $\ov\OO[q,m]_{(p,n)}$) to mean that
 \begin{equation}\label{def:M}
 {M}^{\otimes}:= M_1\otimes \cdots \otimes  M_\ell
\end{equation}
 for some $M_1, \ldots, M_\ell \in \w\OO$ (resp., $\OO_{(p,n)}$ and $\ov\OO[q,m]_{(p,n)}$).}

For $x \in \DG$ (resp., ${\G}_{(p,n)}$ and $\ov{\G}[q,m]_{(p,n)}$)  and $i=1, \ldots, \ell$, let
$$
x^{(i)}=\underbrace{1\otimes \cdots \otimes1\otimes \stackrel{i}{x}\otimes1\otimes \cdots \otimes1}_{\ell}.
$$
For any operator $A=\sum_{r\in I} x_r\otimes y_r$, where $x_r, y_r \in \w{\G}$
(resp., ${\G}_{(p,n)}$ and
$\ov{\G}[q,m]_{(p,n)}$),
and for any $i,j \in \{ 1,\ldots, \ell \}$ with $i \not=j$,
we define
$$
A^{(ij)}=\sum_{r\in I}x_r^{(i)}y_r^{(j)}.
$$
Then ${\w{\Omega}}^{(ij)}$ can be viewed as a linear endomorphism on ${M}^{\otimes}$.
Similarly, we can define the linear operator
${\Omega}^{(ij)}_{(p,n)}$ (resp., $\ov{\Omega}[q,m]_{(p,n)}^{(ij)}$)
on the tensor product $M^{\otimes} \in \OO_{(p,n)}$ (resp., $\ov{\OO}[q,m]_{(p,n)}$) in \eqref{def:M}.

Let ${\bf X}_\ell:=\{(z_1, \ldots,z_\ell)\in \C^\ell\,|\, z_i\not=z_j \,\, \mbox{for any $i\not=j$}\}$
be the \emph{configuration space} of $\ell$ distinct points on $\C^\ell$.
For any $i=1, \ldots, \ell$ and $(z_1, \ldots, z_\ell)\in {\bf X}_\ell$,
the \emph{quadratic Gaudin Hamiltonians} $\wcH^i$, $\cH^i_{(p,n)}$ and $\ovcH^i[q,m]_{(p,n)}$
for $\DG$, $\G_{(p,n)}$ and $\SG[q,m]_{(p,n)}$ are respectively defined by
\begin{equation}\label{qcGH}
\wcH^i:=\sum_{\substack{j=1 \\ j\not=i}}^\ell \frac{\w{\Omega}^{(ij)}}{z_i-z_j}, \quad
\cH^i_{(p,n)}:=\sum_{\substack{j=1 \\ j\not=i}}^\ell \frac{{\Omega}^{(ij)}_{(p,n)}}{z_i-z_j}, \quad
\ovcH^i[q,m]_{(p,n)}:=\sum_{\substack{j=1 \\ j\not=i}}^\ell \frac{{\ov{\Omega}}[q, m]^{(ij)}_{(p,n)}}{z_i-z_j}.
\end{equation}
They are well-defined linear endomorphisms on $M^{\otimes}\in \w\OO$, $\OO_{(p,n)}$ and $\ov\OO[q,m]_{(p,n)}$, respectively.

For any $N \in \w\OO$
(resp., $\OO_{(p,n)}$ and $\ov\OO[q,m]_{(p,n)}$),
let
\begin{equation}\label{def:N^sing}
N^{\rm sing}:=\setc*{ v \in N }{E_{i,j} v=0 \mbox{ for all\,\, $i<j$ \big(resp., $i<j$ and $i <_{\ov{\I}{[q, m]}} j$\big)}\!}
\end{equation}
denote the {\em singular space} of $N$ (with respect to the standard Borel subalgebra $\w{\mf b}$ (resp., ${\mf b}_{(p,n)}$ and $\ov{\mf b}[q,m]_{(p,n)}$)).
Any nonzero vector in $N^{\rm sing}$ is called a singular vector.
For any weight $\mu$ of $N$,
the weight space $N_\mu$ is finite-dimensional by definition of
$\w\OO$ (resp., $\OO_{(p,n)}$ and $\ov\OO[q,m]_{(p,n)}$).
We call
$N_\mu^{\rm sing}:=N_\mu \cap N^{\rm sing}$
the singular weight space of $N$ of weight $\mu$.

The Gaudin Hamiltonians
$\wcH^i$ (resp., $\cH^i_{(p,n)}$  and $\ovcH^i[q,m]_{(p,n)}$) mutually commute with each other,
and they are $\DG$(resp., $\G_{(p,n)}$ and $\SG[q,m]_{(p,n)}$)-homomorphisms for $i=1, \ldots, \ell$ (cf. \cite[Proposition 3.5 and Proposition 3.7]{CaL}).
The space $(M^{\otimes})^{\rm sing}$
and its finite-dimensional subspace
$(M^{\otimes})_\mu^{\rm sing}$
are $\wcH^i$-invariant for any weight $\mu$ of $M^{\otimes}$. Thus, we may view
$\wcH^i$
as a linear endomorphism on
$(M^{\otimes})_\mu^{\rm sing}$.
Similarly,
$\cH^i_{(p,n)}$ (resp., $\ovcH^i[q,m]_{(p,n)}$)
 may be viewed as a linear endomorphism on $(M^{\otimes})_\mu^{\rm sing}$ for $M^{\otimes} \in
 \OO_{(p,n)}$ (resp., $\ov\OO[q,m]_{(p,n)}$) and any weight $\mu$ of $M^{\otimes}$.

\begin{lem} \label{wO-O}
Let $M_1, M_2 \in \w\OO$,
and let $v\in M_1 \otimes M_2$
be a weight vector of weight $\mu$.
\begin{enumerate}[\normalfont(i)]
  \item If $\mu \in \Xi$, then $\w\Omega(v)=\Omega(v)$.
  \item If $\mu \in \ov\Xi[q,m]$, then $\w\Omega(v)=\ov\Omega[q,m](v)$.
\end{enumerate}
\end{lem}
\begin{proof}
Let us prove (ii). Suppose $\mu \in \ov\Xi[q,m]$. We may assume that $v=v_1 \otimes v_2$, where $v_i \in M_i$ is a weight vector of weight $\mu_i$ for $i=1, 2$, and $\mu_1+\mu_2=\mu$.
By \lemref{weight-decomposition2}, $\mu_1, \mu_2 \in \ov\Xi[q,m]$. For any $r \in \hf \Zs \backslash \ov\I[q,m]$, we have $\mu_i(E_r)=0$ and hence $E_r v_i=0$ for $i=1, 2$.
Now given any $r, s \in \hf \Zs$ with $r \notin \ov\I[q,m]$ and $i=1, 2$. If $r<0$, then $E_{r, s} v_i=0$, for otherwise the weight of $E_{r, s} v_i$ would lie outside $\w\Xi$. Similarly, if $r>0$, then $E_{s,r} v_i=0$. In any case, $(E_{r,s} \otimes E_{s,r})v=0$ and $(E_{s,r} \otimes E_{r,s})v=0$. This proves (ii). The proof of (i) is similar.
\end{proof}

\begin{lem} \label{corresp}
Let $M^{\otimes} \in \w\OO$,
and let ${v} \in M^{\otimes}
$
be a weight vector of weight $\mu$.
\begin{enumerate}[\normalfont(i)]
  \item If
$\mu  \in \Xi$,
then
$\wcH^i (v)=\cH^i(v)$ for each $i=1, \ldots, \ell$.

  \item If
$\mu  \in \ov{\Xi}[q,m]$,
then
$\wcH^i (v)=\ovcH^i[q,m](v)$ for each $i=1, \ldots, \ell$.
\end{enumerate}
\end{lem}
\begin{proof}
This is an immediate consequence of \lemref{wO-O}.
\end{proof}

By the same argument as in the proof of \cite[Proposition 4.5]{ChL}, we obtain the following proposition. Recall the bijections of weights given in \eqnref{cP}.

 \begin{prop} \label{prop:isom}
Let $M \in \w\OO$, and let $\wmu \in \w{\cP}^+$ be a weight of $M$.

\begin{enumerate}[\normalfont(i)]
\item There exists $A_\mu \in U (\tilde{\mf l})$ such that the map
\begin{align*}
 {\ft}^{\wmu}: M_{\wmu}^{\rm sing} & \to T(M)_{\mu}^{\rm sing} \\
v & \mapsto A_\mu v
\end{align*}
is a linear isomorphism.
Moreover, there exists $B_\mu\in U (\tilde{\mf l})$
such that the inverse of ${\ft}^{\wmu}$ is given by
$({\ft}^{\wmu})^{-1}(w)=B_\mu w$ for $w \in T(M)_{\mu}^{\rm sing}$.

\item
 There exists $\bar{A}_\mu \in U (\tilde{\mf l})$ such that the map
 \begin{align*}
\ov{\ft}^{\wmu}_{[q,m]} : M_{\wmu}^{\rm sing} & \to \ov{T}_{[q,m]}(M)_{\ovmu[q,m]}^{\rm sing} \\
v & \mapsto \bar{A}_\mu v
\end{align*}
is a linear isomorphism.
Moreover, there exists $\ov B_\mu\in U (\tilde{\mf l})$
such that the inverse of $\ov{\ft}^{\wmu}_{[q,m]}$ is given by
$\big(\ov{\ft}^{\wmu}_{[q,m]} \big)^{-1}(w)=\ov B_{\mu} w$
for $w \in \ov{T}_{[q,m]}(M)_{\ovmu[q,m]}^{\rm sing}$.
\end{enumerate}
\end{prop}

Let $\w{M}^{\otimes}=\w M_1\otimes \cdots \otimes \w M_\ell \in \w\OO$. By \propref{tensor-functor}, we have
 \[
  T( \w{M}^{\otimes})=T(\w M_1)\otimes \cdots \otimes T(\w M_\ell ) \quad \hbox{and} \quad
    \ov T_{[q,m]}( \w{M}^{\otimes})=\ov T_{[q,m]}(\w M_1)\otimes \cdots \otimes \ov T_{[q,m]}(\w M_\ell ).
  \]
\begin{prop} \label{prop:H=}
For  $\w{M}^{\otimes}:=\w M_1\otimes \cdots \otimes \w M_\ell \in \w\OO$, let  $  {M}^{\otimes}= T( \w{M}^{\otimes})$ and $ \ov {M}^{\otimes}=\ov T_{[q,m]}( \w{M}^{\otimes})$. For $\wmu \in \w{\cP}^+$ and each $i=1, \ldots, \ell$, we have:
\begin{enumerate}[\normalfont(i)]
\item The action of  $\cH^i$ on $({M}^{\otimes})^{\rm sing}_{\mu}$ equals ${\ft}^{\wmu}\circ\wcH^i\circ ({\ft}^{\wmu})^{-1}$.
\item The action of  $\ovcH^i[q,m]$ on $(\ov{M}^{\otimes})^{\rm sing}_{\ov\mu[q,m]}$ equals $\ov{\ft}_{[q,m]}^{\wmu}\circ\wcH^i\circ (\ov{\ft}_{[q,m]}^{\wmu})^{-1}$.
\end{enumerate}
\end{prop}

\begin{proof}
  We will only prove (i). The proof of (ii) is similar. For $v\in ({M}^{\otimes})^{\rm sing}_{\mu}$, we have
  $$
  {\ft}^{\wmu}\circ\wcH^i\circ ({\ft}^{\wmu})^{-1}(v)={\ft}^{\wmu}\circ\wcH^i(B_\mu v)={\ft}^{\wmu}(B_\mu \wcH^i(v)) =\wcH^i(v)=\cH^i(v).
  $$
  The last equality follows from \lemref{corresp}.
\end{proof}

 Consequently, we have the following theorem.

\begin{thm} \label{thm:eigenvector}
For $\w{M}^{\otimes}:=\w M_1\otimes \cdots \otimes \w M_\ell \in \w\OO$, let  ${M}^{\otimes}= T( \w{M}^{\otimes})$ and $ \ov {M}^{\otimes}=\ov T_{[q,m]}( \w{M}^{\otimes})$. For $\wmu \in \w{\cP}^+$ and each $i=1, \ldots, \ell$, we have:

\begin{enumerate}[\normalfont(i)]
\item The actions of $\wcH^i$ on $(\w{M}^{\otimes})_{\wmu}^{\rm sing}$ and $\cH^i$ on $({M}^{\otimes})^{\rm sing}_{\mu}$ have the same spectrum, and the map ${\ft}^{\wmu}$ gives a linear isomorphism from the (generalized) eigenspace of $\wcH^i$ on $(\w{M}^{\otimes})_{\wmu}^{\rm sing}$ to the (generalized) eigenspace of $\cH^i$ on $({M}^{\otimes})^{\rm sing}_{\mu}$ for each common eigenvalue $c$.

\item The actions of $\wcH^i$ on $(\w{M}^{\otimes})_{\wmu}^{\rm sing}$ and $\ovcH^i[q,m]$ on $(\ov {M}^{\otimes})^{\rm sing}_{\ovmu[q,m]}$ have the same spectrum, and the map $\ov{\ft}^{\wmu}_{[q,m]}$ gives a linear isomorphism from the (generalized) eigenspace of $\wcH^i$ on $(\w{M}^{\otimes})_{\wmu}^{\rm sing}$ to the (generalized) eigenspace of $\ovcH^i[q,m]$ on $(\ov {M}^{\otimes})^{\rm sing}_{\ovmu[q,m]}$ for each common eigenvalue $c$.
  \end{enumerate}
As a consequence, for each $i$, $\ovcH^i[q,m]$ is diagonalizable on $(\ov {M}^{\otimes})^{\rm sing}_{\ovmu[q,m]}$ if and only if $\cH^i$ is diagonalizable on $({M}^{\otimes})^{\rm sing}_{\mu}$. In this case, they have the same spectrum.
\end{thm}

\begin{rem}
The Jordan canonical forms of $\wcH^i$ on $(\w{M}^{\otimes})^{\rm sing}_{\wmu}$, $\cH^i$ on $({M}^{\otimes})^{\rm sing}_{\mu}$ and $\ovcH^i[q,m]$ on $(\ov {M}^{\otimes})^{\rm sing}_{\ovmu[q,m]}$ are the same for each fixed $i$.
\end{rem}

\subsection{Super Knizhnik-Zamolodchikov equations for $\DG$, $\G_{(p,n)}$ and $\SG[q,m]_{(p,n)}$}\label{sec:ScKZ}

Fix $\ell \in \N$ with $\ell \geq 2$.
For any nonempty open subset $U$ of
${\bf X}_\ell$
and any finite-dimensional vector space $N$,
let $\mc{D}(U, N)$
denote the set of differentiable functions
from $U$ to $N$.
For $N\in \w{\OO}$
(resp., $\OO_{(p,n)}$
and $\ov{\OO}[q,m]_{(p,n)}$),
let
 \[
 \mc{D}(U, N):=\bigoplus_\mu \mc{D}(U, N_\mu),
 \]
 where $\mu$ runs over all weights of $N$. Note that $\mc{D}(U, N)$ is a module over $\DG$ (resp., $\G_{(p,n)}$ and $\SG[q,m]_{(p,n)}$) for $N \in \w{\OO}$ (resp., $\OO_{(p,n)}$ and $\ov{\OO}[q,m]_{(p,n)}$).

For ${M}^{\otimes} \in \w\OO$ (resp., $\OO_{(p,n)}$ and $\ov\OO[q,m]_{(p,n)}$) and $(z_1, \ldots, z_\ell) \in {\bf X}_\ell$,
each quadratic Gaudin Hamiltonian $\wcH^i$ (resp., $\cH^i_{(p,n)}$ and $\ovcH^i[q,m]_{(p,n)}$) is also a linear endomorphism on $\mc{D}(U, {M}^{\otimes})$.

Fix a nonzero $\kappa \in \C$ and ${\psi}(z_1,\ldots, z_\ell)\in \mc{D}(U, M^{\otimes})$.
We consider the system of partial differential equations
\begin{equation}\label{dKZE_n}
\kappa\frac{\partial}{\partial {z_i}}{\psi}(z_1,\ldots,z_\ell)=\w \cH^i{\psi}(z_1,\ldots,z_\ell),\ \ \ \ \mbox{for}\ i=1,\ldots,\ell.
\end{equation}
The equations \eqref{dKZE_n} are called the \emph{super Knizhnik-Zamolodchikov equations
(super KZ equations for short)}
over $\DG$.
Similarly, we consider the
\emph{KZ equations} over $\G_{(p,n)}$
and \emph{super KZ equations}
over $\SG[p,m]_{(p,n)}$ as follows:
\begin{equation}\label{KZE_n}
\kappa\frac{\partial}{\partial {z_i}}{\psi}(z_1,\ldots,z_\ell)=\cH^i_{(p,n)}{\psi}(z_1,\ldots,z_\ell),
\end{equation}
and
\begin{equation}\label{sKZE_n}
\kappa\frac{\partial}{\partial {z_i}}{\psi}(z_1,\ldots,z_\ell)=\ovcH^i[q,m]_{(p,n)}{\psi}(z_1,\ldots,z_\ell),
\end{equation}
for $ i=1,\ldots,\ell$ and ${\psi}(z_1,\ldots, z_\ell)\in \mc{D}(U, {M}^{\otimes})$, where ${M}^{\otimes}\in \OO_{(p,n)}$ (resp., $\ov\OO[q,m]_{(p,n)}$) in \eqref{KZE_n} (resp., \eqref{sKZE_n}).

For any weight $\mu$ of ${M}^{\otimes}\in \w\OO$, let
 \[
 \w{\rm KZ}(({M}^{\otimes})_\mu):=\setc*{\psi\in \mc{D}(U, ({M}^{\otimes})_\mu)}{\mbox{$\psi$ is a solution of the super KZ equations \eqnref{dKZE_n}}\!}
\]
 and
\[\w{\rm KZ}({M}^{\otimes}):=\bigoplus_\mu \w{\rm KZ}(({M}^{\otimes})_\mu)\]
where $\mu$ runs over all weights of ${M}^{\otimes}$.
 We denote
$$
\mc{D}(U, ({M}^{\otimes})_\mu)^{\rm sing}:=\setc*{ \psi  \in \mc{D}(U, ({M}^{\otimes})_\mu)}{E_{i,j} \psi=0 \mbox{ for all $i<j$}}.
$$
The set of \emph{singular solutions} of the (super) KZ equations of weight $\mu$ on ${M}^{\otimes}$ is defined to be
$$\w\cS(({M}^{\otimes})_\mu):=\w{\rm KZ}(({M}^{\otimes})_\mu)\cap \mc{D}(U, ({M}^{\otimes})_\mu)^{\rm sing}.$$
Denote
$$\w\cS({M}^{\otimes}):=\bigoplus_\mu \w\cS(({M}^{\otimes})_\mu)$$
where $\mu$ runs over all weights of ${M}^{\otimes}$.

Similarly, we can define  ${\rm KZ}_{(p,n)}(({M}^{\otimes})_\mu)$, ${\rm KZ}_{(p,n)}({M}^{\otimes})$, $\cS_{(p,n)}(({M}^{\otimes})_\mu)$ and $\cS_{(p,n)}({M}^{\otimes})$ for $M^{\otimes}\in \OO_{(p,n)}$ and define $\ov{\rm KZ}[q,m]_{(p,n)}(({M}^{\otimes})_\mu)$, $\ov{\rm KZ}[q,m]_{(p,n)}({M}^{\otimes})$, $\ov\cS[q,m]_{(p,n)}(({M}^{\otimes})_\mu)$ and $\ov\cS[q,m]_{(p,n)}({M}^{\otimes})$ for $M^{\otimes} \in \ov\OO[q,m]_{(p,n)}$.

Since $\wcH^i$ (resp., $\cH^i_{(p,n)}$  and $\ovcH^i[q,m]_{(p,n)}$) mutually commute with each other and commute with the action of the Lie (super)algebra $\DG$ (resp., $\G_{(p,n)}$ and $\SG[q,m]_{(p,n)}$), the following proposition is clear.

\begin{prop}\label{KZ-g}
The vector space $\w{\rm KZ}({M}^{\otimes})$ (resp., ${\rm KZ}_{(p,n)}({M}^{\otimes})$
and $\ov{\rm KZ}[q,m]_{(p,n)}({M}^{\otimes})$)
is a $\G$(resp., $\G_{(p,n)}$ and $\SG[q,m]_{(p,n)}$)-module for ${M}^{\otimes}\in \w\OO$ (resp., $\OO_{(p,n)}$ and $\ov\OO[q,m]_{(p,n)}$).
\end{prop}

For $\psi(z_1, \ldots, z_\ell) \in \cD(U,M^{\otimes})$, where $M^{\otimes}\in \w\OO$, define $T \psi(z_1, \ldots, z_\ell)  \in \cD(U,T(M^{\otimes}))$ and $\ov T_{[q,m]} \psi(z_1, \ldots, z_\ell)  \in \cD(U,\ov T_{[q,m]}(M^{\otimes}))$ by
$$
T\psi=
\begin{cases}
\psi  &  \mbox{if}\ \ \mu\in \Xi,\\
0     & otherwise,
\end{cases}
\qquad \mbox{and} \qquad
\ov T_{[q,m]} \psi=
\begin{cases}
\psi  &  \mbox{if}\ \ \mu\in \ov{\Xi}[q,m],\\
0     & otherwise.
\end{cases}
$$
Let us drop the subscript $(p,n)$ for $(p,n)=(\infty, \infty)$.

\begin{prop}\label{T-KZ}
Let ${\psi}\in \mc{D}(U, M^{\otimes})$, where $M^{\otimes}\in \w\OO$.
\begin{enumerate}[\normalfont(i)]

\item For $\mu\in \Xi$,
${\psi}\in \w{\rm KZ}(({M}^{\otimes})_\mu)$
if and only if
$T{\psi}\in {\rm KZ}(T({M}^{\otimes})_\mu)$,

\item For $\mu \in \ov\Xi[q,m]$,
${\psi}\in \w{\rm KZ}(({M}^{\otimes})_\mu)$
if and only if
$\ov T_{[q,m]}{\psi} \in \ov{\rm KZ}[q,m](\ov T_{[q,m]}({M}^{\otimes})_\mu)$.

\end{enumerate}
\end{prop}

 \begin{proof}
 This follows from \lemref{weight-decomposition} and \lemref{wO-O}.
\end{proof}

\begin{prop}\label{T-KZ-odd}
With notations as in \propref{prop:isom} and \thmref{thm:eigenvector}. Let $\wmu \in \w{\cP}^+$ be a weight of $\w M ^{\otimes}\in \w\OO$.
 \begin{enumerate}[\normalfont(i)]

 \item If ${\psi}\in  \w{\rm KZ}((\w {M}^{\otimes})_{\wmu})$,
 then
 $A_\mu {\psi} \in  {\rm KZ}(({M}^{\otimes})_\mu)$
 and
 $\bar{A}_\mu {\psi}\in \ov{\rm KZ}[q,m] ((\ov M^{\otimes})_{\ovmu[q,m]})$.

 \item If ${\psi}\in \mc{D}(U, ({M}^{\otimes})_{\mu})$
 and
 $T{\psi}\in {\rm KZ}(({M}^{\otimes})_\mu)$,
 then
 $B_\mu{\psi}\in \w{\rm KZ}((\w{M}^{\otimes})_{\wmu})$.

 \item  \sloppy 
If ${\psi}\in \mc{D}(U,(\ov M^{\otimes})_{\ovmu[q,m]})$
 and
 $\ov T_{[q,m]} {\psi} \in \ov{\rm KZ}[q,m]((\ov {M}^{\otimes})_{\ovmu[q,m]})$, then $\ov{B}_\mu{\psi}\in \w{\rm KZ}((\w{M}^{\otimes})_{\wmu})$.
   \end{enumerate}

 \end{prop}

 \begin{proof}
 This follows from Proposition \ref{prop:isom} and \propref{T-KZ}.
\end{proof}

The following result establishes linear isomorphisms between the sets of singular solutions of (super) KZ equations for $\DG$, $\G$ and $\SG[q,m]$.

\begin{prop} \label{sing-sol}
With notations as in \propref{prop:isom} and \thmref{thm:eigenvector}. Let $\wmu \in \w{\cP}^+$ be a weight of $\w M ^{\otimes}\in \w\OO$.

\begin{enumerate}[\normalfont(i)]
\item The map
\begin{align*}
\w\cS((\w{M}^{\otimes})_{\wmu}) & \longrightarrow \cS((M^{\otimes})_{\mu}) \\
\psi & \mapsto A_\mu \psi
\end{align*}
is a linear isomorphism, which we also denote by ${\ft}^{\wmu}$. The inverse of ${\ft}^{\wmu}$ is given by $({\ft}^{\wmu})^{-1}(\varphi)=B_\mu \varphi$ for $\varphi \in \cS((M^{\otimes})_{\mu})$.

\item The map
\begin{align*}
\w\cS((\w{M}^{\otimes})_{\wmu}) & \longrightarrow \ov\cS[q,m]\big((\ov M^{\otimes})_{\ovmu[q,m]}\big)\\
 \psi & \mapsto \bar{A}_\mu \psi
\end{align*}
is a linear isomorphism, which we also denote by $\ov{\ft}^{\wmu}_{[q,m]}$. The inverse of $\ov{\ft}^{\wmu}_{[q,m]}$ is given by $\big(\ov{\ft}^{\wmu}_{[q,m]} \big)^{-1} (\varphi) = \ov{B}_\mu \varphi$ for $\varphi \in \ov\cS[q,m]\big((\ov M^{\otimes})_{\ovmu[q,m]}\big)$.
\end{enumerate}

\end{prop}

\begin{proof}
This follows from \propref{prop:isom} and \propref{T-KZ-odd}.
\end{proof}

 For the truncation functor
 ${\mf{tr}}^{(p,n)}_{(r,k)}$ (resp., $\ov{\mf{tr}}[q,m]^{(p,n)}_{(r,k)}$) and
 ${\psi}(z_1,\ldots,z_\ell) \in \mc{D}(U,(M^{\otimes})_\mu)$, where $\mu$ is a weight of $M^{\otimes}\in {\mc {O}}_{(p,n)}$
 (resp., $\ov{\mc {O}}[q,m]_{(p,n)}$),
we define the function ${\mf{tr}}^{(p,n)}_{(r,k)}({\psi}) \in \mc{D}\big(U,{\mf{tr}}^{(p,n)}_{(r,k)}(M^{\otimes})_\mu\big)$ (resp., $\ov{\mf{tr}}[q,m]^{(p,n)}_{(r,k)}({\psi}) \in \mc{D}\big(U, \ov{\mf{tr}}[q,m]^{(p,n)}_{(r,k)}(M^{\otimes})_\mu \big)$) by
$$
\mf{tr}^{(p,n)}_{(r,k)}({\psi})=
\begin{cases}
\psi  &  \mbox{if}\ \ \mu\in \Xi_{(r,k)},\\
0     & otherwise,
\end{cases}
$$
$$
\mbox{\Bigg(resp.,}\ \ \ \ \
\ov{\mf{tr}}[q,m]^{(p,n)}_{(r,k)}({\psi})=
\begin{cases}
\psi  &  \mbox{if}\ \ \mu\in \ov{\Xi}[q,m]_{(r,k)},\\
0     & otherwise.
\end{cases}\Bigg)
$$
By Lemma \ref{lem:finite_sum}, we obtain the following.

\begin{prop}\label{tr-KZ}
Let $0 \le r \le p \leq\infty$, $0<k \le n \leq\infty$.

\begin{enumerate}[\normalfont(i)]
\item For $\mu\in \Xi_{(r,k)}$, $M^{\otimes}\in{\mc {O}}_{(p,n)}$  and ${\psi}\in \mc{D}(U, (M^{\otimes})_\mu)$, we have
\[
\quad
{\psi}\in {\rm KZ}_{(p,n)}((M^{\otimes})_\mu) \Leftrightarrow \mf{tr}^{(p,n)}_{(r,k)}({\psi})\in {\rm KZ}_{(r,k)}\big(\mf{tr}^{(p,n)}_{(r,k)}(M^{\otimes})_\mu\big)\]
and
\[{\psi}\in \cS_{(p,n)}((M^{\otimes})_\mu) \Leftrightarrow \mf{tr}^{(p,n)}_{(r,k)}({\psi})\in \cS_{(r,k)}\big(\mf{tr}^{(p,n)}_{(r,k)}(M^{\otimes})_\mu\big).
\]

\item For $\mu\in \ov\Xi[q,m]_{(r,k)}$, $M^{\otimes}\in \ov{\mc {O}}[q,m]_{(p,n)}$, and ${\psi}\in \mc{D}(U, (M^{\otimes})_\mu)$,
we have
\[
\quad
{\psi}\in \ov{\rm KZ}[q,m]_{(p,n)}((M^{\otimes})_\mu) \Leftrightarrow \ov{\mf{tr}}[q,m]^{(p,n)}_{(r,k)}({\psi})\in \ov{\rm KZ}[q,m]_{(r,k)}\big(\ov{\mf{tr}}[q,m]^{(p,n)}_{(r,k)}(M^{\otimes})_\mu\big)\]
and
\[{\psi}\in \ov\cS[q,m]_{(p,n)}((M^{\otimes})_\mu) \Leftrightarrow \ov{\mf{tr}}[q,m]^{(p,n)}_{(r,k)}({\psi})\in
\ov\cS[q,m]_{(r,k)}\big(\ov{\mf{tr}}[q,m]^{(p,n)}_{(r,k)}(M^{\otimes})_\mu\big).
\]
\end{enumerate}
\end{prop}

\section{Quadratic Gaudin Hamiltonians and super KZ equations for $\cG_{(p,n)}$ and $\ovcG[q,m]_{(p,n)}$} \label{sec:quad-cG}

In this section, we study the quadratic Gaudin Hamiltonians for the general linear Lie superalgebra $\ovcG[q,m]_{(p,n)}$ on the tensor product of unitarizable highest weight modules and their diagonalization.

On the other hand, we relate the set of singular solutions of the (super) KZ equations for $\ovcG[q,m]_{(p,n)}$ to the set of singular solutions of the KZ equations for $\cG_{(r,k)}$ for $r$ and $k$ sufficiently large.

\subsection{Quadratic Gaudin Hamiltonians for $\cG_{(p,n)}$ and $\ovcG[q,m]_{(p,n)}$}
Fix $m,p,q \in \Zp$ and $n \in \N$.
The Casimir symmetric tensor $\mathring{\Omega}_{(p,n)}$ for $\cG_{(p,n)}$ is defined by
$$
\mathring{\Omega}_{(p,n)}:=\sum_{i,j\in {\I}_{(p,n)}} (-1)^{2j} E_{i,j} \otimes E_{j,i}.
$$
The Casimir symmetric tensor $\mathring{\ov{\Omega}}[q,m]_{(p,n)}$ for $\ovcG[q,m]_{(p,n)}$ is defined similarly with $\I_{(p,n)}$ being replaced with  $\ov{\I}{[q, m]}_{(p, n)}$.
Clearly, $\mathring{\Omega}_{(p,n)}$ and $\mathring{\ov{\Omega}}[q,m]_{(p,n)}$ lie in $U(\cG_{(p,n)}) \otimes U(\cG_{(p,n)})$ and $U(\ovcG[q,m]_{(p,n)}) \otimes U(\ovcG[q,m]_{(p,n)})$, respectively.
A straightforward calculation shows that
\begin{eqnarray} \label{Om-Om}
  \iota\otimes \iota \big(\mathring{\Omega}_{(p,n)}\big)
  &=&\Omega_{(p,n)}-pK\otimes K, \\
  \iota\otimes \iota \big(\mathring{\ov\Omega}[q,m]_{(p,n)}\big)
  &=&\ov\Omega[q,m]_{(p,n)}+(p-q)K\otimes K. \nonumber
\end{eqnarray}
where $\iota$ are the isomorphisms defined in \eqnref{iso-e}.

Fix $\ell \in \N$ with $\ell \geq 2$.
For any $i=1, \ldots, \ell$ and $(z_1, \ldots, z_\ell)\in {\bf X}_\ell$,
the \emph{quadratic Gaudin Hamiltonian}s
$H^i_{(p,n)}$ for ${\cG}_{(p,n)}$
and $\ov H^i [q,m]_{(p,n)}$ for ${\ovcG}[q,m]_{(p,n)}$ are defined by
\begin{equation} \label{qGH}
 H^i_{(p, n)}:=\sum_{\substack{j=1\\ j\neq i}}^\ell\frac{\mathring{\Omega}^{(ij)}_{(p,n)}}{z_i-z_j} \quad \mbox{and} \quad \ov H^i [q,m]_{(p,n)}:=\sum_{\substack{j=1\\ j\neq i}}^\ell\frac{\mathring{\ov{\Omega}}[q,m]_{(p,n)}^{(ij)}}
 {z_i-z_j}.
\end{equation}

For any $\cG_{(p,n)}$(resp., $\ovcG[q,m]_{(p,n)}$)-module $N$, let
\begin{equation}\label{def:sing}
N^{\rm sing}:=\setc*{ v \in N }{E_{i,j} v=0 \mbox{ for all\,\, $i<j$ \big(resp., $i <_{\ov{\I}{[q, m]}} j$\big)}\!}
\end{equation}
denote the {\em singular space} of $N$ (with respect to the standard Borel subalgebra ${\bf b}_{(p,n)}$ (resp., $\ov{{\bf b}}[q,m]_{(p,n)}$)).
Any nonzero vector in $N^{\rm sing}$ is called a singular vector.
For any weight $\mu$ of $N$, we call
$N_\mu^{\rm sing}:=N_\mu \cap N^{\rm sing}$
the singular weight space of $N$ of weight $\mu$.

Let ${M}^{\otimes}$ (resp., $\ov M^{\otimes}$) be a $\cG_{(p,n)}$(resp., $\ovcG[q,m]_{(p,n)}$)-module defined as in \eqnref{def:M}. The Hamiltonians $H^i_{(p, n)}$ (resp., $\ov H^i [q,m]_{(p,n)}$) are linear endomorphisms on ${M}^{\otimes}$ (resp., $\ov M^{\otimes}$). Similar to the cases for central extensions, $H^i_{(p,n)}$ (resp., $\ov{H}^i[q,m]_{(p,n)}$) on ${M}^{\otimes}$ (resp., $\ov M^{\otimes}$) mutually commute with each other, and they commute with the action of $\cG_{(p,n)}$ (resp., $\ovcG[q,m]_{(p,n)}$).
It is clear that for any weight $\mu$ of ${M}^{\otimes}$ (resp., $\ov M^{\otimes}$), the subspace $({M}^{\otimes})_{\mu}^{\rm sing}$ (resp., $(\ov M^{\otimes})_{\mu}^{\rm sing})$) is ${H}^i_{(p,n)}$(resp., $\ov{H}^i[q,m]_{(p,n)}$)-invariant.

Each $\G_{(p,n)}$(resp., $\SG[q,m]_{(p,n)}$)-module may be regarded as a $\cG_{(p,n)}$(resp., $\ovcG[q,m]_{(p,n)}$)-module via the isomorphism $\iota$ defined in \eqref{iso-e}. Clearly, the set of singular vectors in a module $N$ regarded as a $\G_{(p,n)}$(resp., $\SG[q,m]_{(p,n)}$)-module equals the set of singular vectors in the module $N$ regarded as a $\cG_{(p,n)}$(resp., $\ovcG[q,m]_{(p,n)}$)-module.

The quadratic Gaudin Hamiltonians for $\cG_{(p,n)}$ (resp., $\ovcG[q,m]_{(p,n)}$) and those for $\G_{(p,n)}$ (resp., $\SG[q,m]_{(p,n)}$) are related as follows.

\begin{prop}\label{prop:iota_corresp}
Let  $M^{\otimes}$ be a $\G_{(p,n)}$-module  and $\ov M^{\otimes}$ a $\SG[q,m]_{(p,n)}$-module. For $i=1, \ldots, \ell$, we have:

\begin{enumerate}[\normalfont(i)]

\item $v \in (M^{\otimes})^{\rm sing}$ is a (generalized) eigenvector of $\cH^i_{(p,n)}$ if and only if $v$ is a (generalized) eigenvector of $H^i_{(p,n)}$. Consequently, $\cH^i_{(p,n)}$ is diagonalizable on $(M^{\otimes})^{\rm sing}$  if and only if $H^i_{(p,n)}$ is diagonalizable on $(M^{\otimes})^{\rm sing}$.

\item $\ov v \in (\ov M^{\otimes})^{\rm sing}$  is a (generalized) eigenvector of $\ovcH^i[q,m]_{(p,n)}$ if and only if $\ov v$ is a (generalized) eigenvector of $\ov H^i[q,m]_{(p,n)}$. Consequently, $\ovcH^i[q,m]_{(p,n)}$ is diagonalizable on $(\ov M^{\otimes})^{\rm sing}$ if and only if $\ov H^i[q,m]_{(p,n)}$ is diagonalizable on $(\ov M^{\otimes})^{\rm sing}$.

\end{enumerate}

\end{prop}

\begin{proof}
We only prove (ii). The proof of (i) is similar.
By \eqref{Om-Om}, we find that
$$
\ov H^i[q,m]_{(p,n)}(\ov v)=\ov \cH^i[q,m]_{(p,n)}(\ov v)+(p-q)\sum_{\substack{j=1 \\ j\not=i}}^\ell \frac{K^{(i)}  K^{(j)}}{z_i-z_j}(\ov v).
$$
The last term on the right hand side is clearly a fixed scalar multiple of $\ov v$. This proves (ii).
\end{proof}

\subsection{Diagonalization and correspondences between eigenbases}

We set
$$
L(\cG_{(p,n)}, \un{\xi}):=L(\cG_{(p,n)}, \xi_1) \otimes \cdots \otimes L(\cG_{(p,n)}, \xi_\ell)
$$
$$
\left(\textrm{resp., }L(\ovcG[q,m]_{(p,n)}, \un{\ov{\xi}}):=L(\ovcG[q,m]_{(p,n)}, \ov\xi_1) \otimes \cdots \otimes L(\ovcG[q,m]_{(p,n)}, \ov\xi_\ell) \right).
$$
for $\un{{\xi}}:=(\xi_1, \ldots, \xi_\ell) \in ({\bf h}_{(p, n)}^*)^{\times \ell}$ (resp., $\un{\ov{\xi}}:=(\ov\xi_1, \ldots, \ov\xi_\ell) \in (\ov{\bf h}[q, m]_{(p, n)}^*)^{\times \ell}$).
Here, e.g., $({\bf h}_{(p, n)}^*)^{\times \ell}$ denotes the Cartesian product of $\ell$ copies of ${\bf h}_{(p, n)}^*$, and recall that $L(\cG_{(p,n)}, \xi_i)$ (resp., $L(\ovcG[q,m]_{(p,n)}, \ov\xi_i)$)
is the irreducible highest weight $\cG_{(p,n)}$(resp., $\ovcG[q,m]_{(p,n)}$)-module with highest weight $\xi_i$ (resp., $\ov\xi_i$) for $i=1, \ldots, \ell$.

Similarly, we set
$$
L(\G_{(p,n)}, \un{\xi}):=L(\G_{(p,n)}, \xi_1) \otimes \cdots \otimes L(\G_{(p,n)}, \xi_\ell)
$$
$$
\left(\textrm{resp., }L(\SG[q,m]_{(p,n)}, \un{\ov{\xi}}):=L(\SG[q,m]_{(p,n)}, \ov\xi_1) \otimes \cdots \otimes L(\SG[q,m]_{(p,n)}, \ov\xi_\ell) \right).
$$
for $\un{{\xi}}:=(\xi_1, \ldots, \xi_\ell)  \in ({\fh}_{(p, n)}^*)^{\times \ell}$ (resp., $\un{\ov{\xi}}:=(\ov\xi_1, \ldots, \ov\xi_\ell) \in (\ov{\fh}[q, m]_{(p, n)}^*)^{\times \ell}$).
The above notations are also used later if $p$ or $n$ equals $\infty$.

Recall that $\cP_d$ denotes the set of all generalized partitions of depth $d$.
We define
\begin{align}
\w{Q}&=\setc*{\wla \in \w{\cP}^+(d)}{ \la \in \cP_d, \, d \in \N },\nonumber \\
Q&=\setc*{\la \in \cP^+(d)}{ \la \in  \cP_d, \, d \in \N  }, \label{Qqm} \\
\ov{Q}[q,m]&=\setc*{\ovla[q,m]\in \ov{\cP}[q,m]^+(d)}{ \la  \in \cP_d, \, d \in \N}.\nonumber
\end{align}
Here $\wla \in \w{\cP}^+(d)$, $\la \in \cP^+(d)$ and $\ovla[q,m] \in \ov{\cP}[q,m]^+(d)$ are defined as in \eqref{weight:wtIm}, \eqref{weight:m} and \eqref{weight:Im}.
It is easy to see that the bijections in \eqnref{cP} restrict to the following bijections:
\begin{equation}\label{Q-bij}
   Q \leftrightarrow  \w Q \leftrightarrow \ov{Q}[q,m].
\end{equation}
Thus every weight $\xi \in Q$ corresponds to a weight $\ov\xi\in \ov{Q}[q,m]$, and vice versa. By \cite[Theorem 5.7]{LZ06}, $L(\DG, \xi)$ is unitarizable if and only if  $\xi \in \w{Q}$. It follows that  $L(\G, \xi)$ (resp., $L(\SG[q,m], \xi)$) is unitarizable if and only if  $\xi \in {Q}$ (resp., $\ov{Q}[q,m]$).

For $(z_1, \ldots, z_\ell)\in {\bf X}_\ell$,
the
\emph{Gaudin algebra} of
 $\cG_{(p,n)}$
 is a commutative subalgebra of
 $U(\cG_{(p,n)})^{\otimes \ell}$
 generated by quadratic Gaudin Hamiltonians as well as higher Gaudin Hamiltonians for
 $\cG_{(p,n)}$ (cf. \cite{FFR, Ta}).
 It acts on
 $M_1 \otimes \cdots \otimes M_\ell$,
 where all $M_i$'s are $\cG_{(p,n)}$-modules.
 The following statement is well-known
 (see \cite[Section 6.1]{MTV09} and \cite[Main Corollary]{Ryb}).

\begin{prop} \label{Gaudin-alg}
Let $\un{{\xi}} \in ({\bf h}_{(p, n)}^*)^{\times \ell}$ be a sequence of dominant integral weights. For a generic $(z_1, \ldots, z_\ell)\in \C^\ell$, every element in the Gaudin algebra of $\cG_{(p,n)}$ is diagonalizable on the space $L(\cG_{(p,n)}, \un{\xi})^{\rm sing}$.
\end{prop}

\begin{cor}\label{cor:quad-diag-central}
Let $\un{{\xi}} \in Q^{\times \ell}$.
For a generic $(z_1, \ldots, z_\ell)\in \C^\ell$,
the quadratic Gaudin Hamiltonians
$\cH^1,\ldots,\cH^\ell$
are simultaneously diagonalizable on
$L(\G, \un{\xi})^{\rm sing}$.
\end{cor}
\begin{proof}
Given any weight space of $L(\G, \un{\xi})$, we may take $p$ and $n$ large enough that it is contained in $\mf{tr}^{(\infty, \infty)}_{(p,n)}(L(\G, \un{\xi}))$.
Also,
$\mf{tr}^{(\infty, \infty)}_{(p,n)}(L(\G,{\xi}_i))$, regarded as a $\cG_{(p,n)}$-module via the isomorphism $\iota$ defined in \eqref{iso-e}, is irreducible with dominant integral highest weight, for each $i=1, \ldots, \ell$. Therefore the statement is a consequence of \propref{prop:iota_corresp} and \propref{Gaudin-alg}.
\end{proof}

\begin{thm}\label{diag-central}
Let $\un{\ov{\xi}} \in  \ov{Q}[q,m]^{\times \ell}$, and let $\un{\xi} \in {Q}^{\times \ell}$ be the corresponding sequence of weights under the bijections \eqref{Q-bij}. Let $\ovmu \in \ov{Q}[q,m]$, and let $\mu \in Q$ be the corresponding weight.
For a generic $(z_1, \ldots, z_\ell)\in \C^\ell$, we have:

\begin{enumerate}[\normalfont(i)]

\item The quadratic Gaudin Hamiltonians
$\ovcH^1[q,m],\ldots,\ovcH^{\ell}[q,m]$
are simultaneously diagonalizable on $L(\SG[q,m], \un{\ov{\xi}})^{\rm sing}_{\ovmu}$.

\item Every joint eigenbasis of $\ovcH^1[q,m],\ldots,\ovcH^{\ell}[q,m]$ on
    $L(\SG[q,m], \un{\ov{\xi}})^{\rm sing}_{\ovmu}$
    can be obtained from some joint eigenbasis of
    $\cH^1,\ldots,\cH^{\ell}$
    on $L(\G, \un{\xi})^{\rm sing}_\mu$, and vice versa.

\item For each $i=1,\ldots,\ell$, the actions of $\ovcH^i[q,m]$ on $L(\SG[q,m], \un{\ov{\xi}})^{\rm sing}_{\ovmu}$ and $\cH^i$ on $L(\G, \un{\xi})^{\rm sing}_\mu$ have the same spectrum.

\end{enumerate}

\end{thm}

\begin{proof}
This follows from \thmref{thm:eigenvector} and \corref{cor:quad-diag-central}.
\end{proof}

For $\un{\ov\xi} \in \ov\cQ[q,m]_{(p,n)}^{\times \ell}$, the module $L(\ovcG[q,m]_{(p,n)}, \un{\ov\xi})$ is a direct sum of irreducible highest weight modules as it is a tensor product of unitarizable highest weight modules. Thus to diagonalize the quadratic Gaudin Hamiltonian $\ov H^i[q,m]_{(p,n)}$ on the space $L(\ovcG[q,m]_{(p,n)}, \un{\ov \xi})$, it suffices to diagonalize $\ov H^i[q,m]_{(p,n)}$ on the subspace $L(\ovcG[q,m]_{(p,n)}, \un{\ov\xi})^{\rm sing}$ since $\ov H^i[q,m]_{(p,n)}$ commutes with the action of $\ovcG[q,m]_{(p,n)}$.

\begin{thm} \label{thm:quad-diag}
Let $\un{\ov\xi}=(\ov \xi_1, \ldots, \ov \xi_\ell)  \in \ov\cQ[q,m]_{(p,n)}^{\times \ell}$ and $\ov\mu\in \ov\cQ[q,m]_{(p,n)}$.
For a generic $(z_1, \ldots, z_\ell)\in \C^\ell$,
the quadratic Gaudin Hamiltonians
$\ov H^1[q,m]_{(p,n)},\ldots, \ov H^{\ell}[q,m]_{(p,n)}$
are simultaneously diagonalizable on the space
$L(\ovcG[q,m]_{(p,n)}, \un{\ov\xi})^{\rm sing}_{\ov\mu}$,
and each joint eigenbasis of the Hamiltonians can be obtained from some joint eigenbasis of the quadratic Gaudin Hamiltonians
$H^1_{(r, k)},\ldots,H^{\ell}_{(r, k)}$ on the corresponding singular weight space of the tensor product of the corresponding finite-dimensional irreducible $\cG_{(r, k)}$-modules
for $r$ and $k$ sufficiently large.
\end{thm}

\begin{proof}
For $i=1, \ldots, \ell$, $\ov \xi_i=\ov{\la^i}-d_i \mathbbm{1}_{p|q}$ for some generalized partition $\la^i$ of depth $d_i$ such that $(\la^i)^+_{m+1}\le n$ and $(\la^i)^-_{p+1}\le q$ (see Section \ref{sec:UM} for related notations).
Note that $L(\ovcG[q,m]_{(p,n)}, \un{\ov\xi})$, regarded as a $\SG[q,m]_{(p,n)}$-module via the isomorphism $\iota$ in \eqnref{iso-e}, is isomorphic to $L(\SG[q,m]_{(p,n)}, \ov{\la^1}[q, m])\otimes\cdots \otimes L(\SG[q,m]_{(p,n)}, \ov{\la^\ell}[q, m])$, where $\ov{\la^i}[q,m] :=\ov{\la^i}+ d_i \La_0  \in \ov{Q}[q,m]$ for $i=1, \ldots, \ell$.
By \propref{prop:iota_corresp} and \thmref{diag-central}, the statement follows.
\end{proof}

\thmref{thm:quad-diag} ensures the existence of a joint eigenbasis of the quadratic Gaudin Hamiltonians for $\ovcG[q,m]_{(p,n)}$ on any singular weight space of $L(\ovcG[q,m]_{(p,n)}, \un{\ov\xi})$ for $\un{\ov{\xi}} \in \ov\cQ[q,m]_{(p,n)}^{\times \ell}$.
However, writing an explicit formula for such an eigenbasis is, in general, rather involved and cumbersome. For an example of how the theorem works (in a simpler setting), see \cite[pp. 429--430]{ChL}.

For $p=q=0$, Mukhin, Tarasov, and Varchenko have shown that the Gaudin algebra of $\cG_{(0, n)}$ on the singular space of the tensor power of the natural module is generated by the quadratic Gaudin Hamiltonians (cf. \cite{MTV10}). This indicates that we can say more about the actions of the quadratic Gaudin Hamiltonians for $\ovcG[0,m]_{(0,n)}$ on the corresponding singular space. We will discuss this in Section \ref{sec:cubic} as the settings there can give a simpler treatment.

\subsection{Super Knizhnik-Zamolodchikov equations for $\cG_{(p,n)}$ and $\ovcG[q,m]_{(p,n)}$}\label{sec:SKZ}

 We fix $m, p, q \in \Zp$ and $n \in \N$.
 Let $\mth_{(p,n)}$
(resp., $\msh[q,m]_{(p,n)}$)
and $\mtb_{(p,n)}$
(resp., $\msb[q,m]_{(p,n)}$)
 denote the Cartan subalgebra
 and the Borel subalgebra of
$\cG_{(p,n)}\oplus \C K$
 (resp., $\ovcG[q,m]_{(p,n)}\oplus \C K$)
 associated to the Dynkin diagram in Section \ref{Dynkin}, respectively.
Let
$\mtl_{(p,n)}$
(resp., $\msl[q,m]_{(p,n)}$)
be the standard Levi subalgebra of
$\cG_{(p,n)}\oplus \C K$
(resp., $\ovcG[q,m]_{(p,n)}\oplus \C K$)
associated to
$Y_{(p,n)}$ (resp., $\ov Y[q,m]_{(p,n)}$) defined in \eqnref{Levi},
and let
$\mtp_{(p,n)}=\mtl_{(p,n)}+\mtb_{(p,n)}$
(resp., $\msp[q,m]_{(p,n)}=\msl[q,m]_{(p,n)}+\msb[q,m]_{(p,n)}$)
 be the corresponding parabolic subalgebra.

For $\mu \in \mth_{(p,n)}^*$ (resp., $\msh[q,m]_{(p,n)}^*$),
let
$L(\mtl_{(p,n)}, \mu)$ (resp., $ L(\msl[q,m]_{(p,n)}, \mu)$)
be the irreducible highest weight module over
$\mtl_{(p,n)}$ (resp., $\msl[q,m]_{(p,n)}$) with highest weight $\mu$.
 The parabolic Verma $\cG_{(p,n)}\oplus \C K$-module is defined to be
 $\mathring{{\Delta}}_{(p,n)}(\mu)
 :=\mbox{Ind}^{\cG_{(p,n)}\oplus \C K}_{\mtp_{(p,n)}} L(\mtl_{(p,n)}, \mu)$.
The parabolic Verma $\ovcG[q,m]_{(p,n)}\oplus \C K$-module $\mathring{\ov{\Delta}}[q,m]_{(p,n)}(\mu)$ is defined similarly.
 The unique irreducible quotient
$\cG_{(p,n)}\oplus \C K$(resp., $\ovcG[q,m]_{(p,n)}\oplus \C K$)-module of
$\mathring{\Delta}_{(p,n)}(\mu)$ (resp., $\mathring{\ov{\Delta}}[q,m]_{(p,n)}(\mu)$)
 is denoted by
$\mathring{L}_{(p,n)}(\mu)$ (resp., $\mathring{\overline L}[q,m]_{(p,n)}(\mu)$).

Since the Cartan subalgebras
$\mth_{(p,n)}$ and ${\h}_{(p,n)}$
(resp., $\msh[q,m]_{(p,n)}$ and $\ov{\h}[q,m]_{(p,n)}$)
coincide,
we may identify the dual space of
$\mth_{(p,n)}$
(resp., $\msh[q,m]_{(p,n)}$)
with that of
${\h}_{(p,n)}$ (resp., $\ov{\h}[q,m]_{(p,n)}$).
Associated to partitions
$\la^+=(\la^+_1,\la^+_2,\ldots)$ and $\la^-=(\la^-_1,\la^-_2,\ldots)$
as well as
$d \in \C$, we define
\begin{equation}\label{weight:Im0}
 \mathring\la := -\sum_{r=1}^p \big( (\la^-)'_r-d \big) \ep_{-r+\hf} +\sum_{i=1}^n (\la^+)'_i \ep_{i-\hf}+d \La_0\in \mth_{(p,n)}^{*}
\end{equation}
for $(\la^+)'_{n+1}=(\la^-)'_{p+1}=0$, and
\begin{align}\label{weight:ovIm0}
\mathring{\ov\la}[q,m] :=&-\sum_{r=1}^p \big(\langle \la^{-}_{r}-q \rangle+d \big) \ep_{-r}- \sum_{s=1}^{q} ((\la^{-})^\prime_{s}-d) \ep_{-s+\hf}  \\  \nonumber
&\hspace{5mm}+\sum_{i=1}^{m} \la^+_{i}\ep_{i}+ \sum_{j=1}^n \langle (\la^+)'_{j}-m \rangle\ep_{j-\hf}+ d\La_0 \in \msh[q,m]_{(p,n)}^*
\end{align}
for $(\la^+)'_{n+1}\leq m$ and $\la^{-}_{p+1}\leq q$.

Let $\mathring{P}_{(p,n)}^{\raisebox{-10pt}{\scriptsize{+}}}(d) \subseteq\mth_{(p,n)}^*$ and
$\mathring{\ov{P}}[q,m]_{(p,n)}^{\raisebox{-5pt}{\scriptsize{+}}}(d)  \subseteq\msh[q,m]_{(p,n)}^*$
 denote the sets of all weights of the forms
\eqnref{weight:Im0} and \eqnref{weight:ovIm0}, respectively.
Let
$$\mathring{{P}}_{(p,n)}^{\raisebox{-10pt}{\scriptsize{+}}}:=
 \bigcup_{d\in \C}\mathring{{P}}_{(p,n)}^{\raisebox{-10pt}{\scriptsize{+}}}(d)
 \ \ \ \ \ \  \ \
 \mbox{and} \ \ \ \ \ \ \  \ \mathring{\ov{P}}[q,m]^{\raisebox{-5pt}{\scriptsize{+}}}_{(p,n)}:=
 \bigcup_{d\in \C}\mathring{\ov{P}}[q,m]^{\raisebox{-5pt}{\scriptsize{+}}}_{(p,n)}(d).$$

 Let $\mathring\OO_{(p,n)}(d)$ (resp., $\mathring{\ov\OO}[q,m]_{(p,n)}(d)$)
 be the category of
$\cG_{(p,n)}\oplus \C K$(resp., $\ovcG[q,m]_{(p,n)}\oplus \C K$)-modules $M$
 such that $M$ is a semisimple $\mth_{(p,n)}$(resp., $\msh[q,m]_{(p,n)}$)-module
 with finite-dimensional weight subspaces
 $M_{\gamma}$, for $\gamma \in \mth_{(p,n)}^*$ (resp., $\msh[q,m]_{(p,n)}^*$),
 satisfying
   \begin{enumerate}[\normalfont(i)]
     \item[ (i)] $M$ decomposes over $\mtl_{(p,n)}$
         (resp., $\msl[q,m]_{(p,n)}$)
     as a direct sum of irreducible modules
 $L(\mtl_{(p,n)}, \mu)$
     (resp., $L(\msl[q,m]_{(p,n), \mu)}$)
     for
  $\mu \in \mathring{{P}}_{(p,n)}^{\raisebox{-10pt}{\scriptsize{+}}}(d)$
   (resp.,  $\mathring{\ov{P}}[q,m]_{(p,n)}^{\raisebox{-5pt}{\scriptsize{+}}}(d)$). \item [(ii)]
 There exist finitely many weights
  $\la^1,\ldots,\la^k$ in the set $\mathring{{P}}_{(p,n)}^{\raisebox{-10pt}{\scriptsize{+}}}(d)$
     (resp., $\mathring{\ov{P}}[q,m]^{\raisebox{-5pt}{\scriptsize{+}}}_{(p,n)}(d)$)
      (depending on $M$) such that if $\gamma$ is a weight of $M$,
      then $\la^i-\gamma$ is a linear combination of simple roots with coefficients in $\Zp$ for some $i$.
   \end{enumerate}

Also, we define $\mathring{{\OO}}_{(p,n)}:=\bigoplus_{d\in \C}\mathring{\OO}_{(p,n)}(d)$ and
$\mathring{\ov\OO}[q,m]_{(p,n)}:=\bigoplus_{d\in \C}\mathring{\ov\OO}[q,m]_{(p,n)}(d)$.
The morphisms in the categories are even homomorphisms of modules.

For $d \in \C$, let ${\OO}_{(p,n)}(d)$ be the full subcategory of ${\OO}_{(p,n)}$ whose objects are those $M \in {\OO}_{(p,n)}$ such that $Kv=dv$ for all $v \in M$.
The category $\ov\OO[q,m]_{(p,n)}(d)$ is defined in the same way. We have ${{\OO}}_{(p,n)}=\bigoplus_{d\in \C}{\OO}_{(p,n)}(d)$ and ${\ov\OO}[q,m]_{(p,n)}=\bigoplus_{d\in \C}{\ov\OO}[q,m]_{(p,n)}(d)$.
Forgetting $\Z_2$-gradations,
we see that there exists an isomorphism
$\Psi: \mathring{\OO}_{(p,n)}(d) \to {\OO}_{(p,n)}(d)$
\big(resp., $\Psi:\mathring{\ov\OO}[q,m]_{(p,n)}(d) \to {\ov\OO}[q,m]_{(p,n)}(d)$\big)
of categories induced from the isomorphism $\iota$ defined in \eqnref{iso-e},
and hence
$\mathring{\OO}_{(p,n)}$ and ${\OO}_{(p,n)}$
(resp., $\mathring{\ov\OO}[q,m]_{(p,n)}$ and ${\ov\OO}[q,m]_{(p,n)}$)
are isomorphic as tensor categories.
For each $M\in\mathring{\OO}_{(p,n)}(d)$ (resp., $\mathring{\ov\OO}[q,m]_{(p,n)}(d)$),
we have $\Psi(M)=M$. Thus $M$ carries a $\Z_2$-gradation inherited from the $\Z_2$-gradation on $\Psi(M)$.

  For $\mu \in \mathring{{P}}_{(p,n)}^{\raisebox{-10pt}{\scriptsize{+}}}(d)$
     (resp., $\mathring{\ov{P}}[q,m]^{\raisebox{-5pt}{\scriptsize{+}}}_{(p,n)}(d)$),
   the parabolic Verma module $\mathring{\Delta}_{(p,n)}(\mu)$ (resp., $\mathring{\ov{\Delta}}[q,m]_{(p,n)}(\mu)$) as well as the
 irreducible module
 $\mathring{L}_{(p,n)}(\mu)$ (resp., $\mathring{\overline L}[q,m]_{(p,n)}(\mu)$) lie in $\mathring{\OO}_{(p,n)}$ (resp., $\mathring{\ov\OO}[q,m]_{(p,n)}$).

 In this subsection, we use the symbol $M^{\otimes}\in\mathring{\OO}_{(p,n)}$ (resp., $\mathring{\ov\OO}[q,m]_{(p,n)}$) to mean that
 $$
 {M}^{\otimes}:= M_1\otimes \cdots \otimes  M_\ell
$$
 for some $M_1, \ldots, M_\ell \in \mathring{\OO}_{(p,n)}$ (resp., $\mathring{\ov\OO}[q,m]_{(p,n)}$).

Fix a nonzero $\kappa \in \C$ and ${\psi}(z_1,\ldots, z_\ell)\in \mc{D}(U, M^{\otimes}):=\bigoplus_\mu \mc{D}(U, (M^{\otimes})_\mu)$, where $\mu$ runs over all weights of
$M^{\otimes}\in\mathring{\OO}_{(p,n)}$ (resp., $\mathring{\ov\OO}[q,m]_{(p,n)}$).
For $i=1,\ldots,\ell$, the quadratic Gaudin Hamiltonian $H^i_{(p,n)}$ (resp., $\ov H^i[q,m]_{(p,n)}$) acts on $\mc{D}(U, M^{\otimes})$, and we may consider the \emph{(super) KZ equations} over ${\cG}_{(p,n)}$ (resp., ${\ovcG}[q,m]_{(p,n)}$): For $i=1,\ldots,\ell$,
\begin{equation}
\kappa\frac{\partial}{\partial {z_i}}{\psi}=H^i_{(p,n)}\psi \quad
\left(\mbox{resp., $\quad \displaystyle{\kappa\frac{\partial}{\partial {z_i}}{\psi}=\ov H^i[q,m]_{(p,n)} \psi}$}\right).
\end{equation}

We let
$\mathring{{\rm KZ}}_{(p,n)}(M^{\otimes})$ and $\mathring{\ov{\rm KZ}}[q,m]_{(p,n)}(M^{\otimes})$ (resp., $\mathring{\cS}_{(p,n)}(M^{\otimes})$ and $\mathring{\ov\cS}[q,m]_{(p,n)}(M^{\otimes})$) be the sets of all solutions (resp., singular solutions) of the (super) KZ equations in $\mc{D}(U, M^{\otimes})$, for ${M}^\otimes \in \mathring{\OO}_{(p,n)}$ (resp., $\mathring{\ov\OO}[q,m]_{(p,n)}$).
 We obtain the following result by a direct computation using \eqnref{Om-Om}.

\begin{prop} \label{ncetoce}
Let $M^{\otimes}=M_1\otimes \cdots \otimes  M_\ell$ and $\psi(z_1,\ldots,z_\ell)\in \mc{D}(U, M^{\otimes})$.
\begin{enumerate}[\normalfont(i)]

  \item Suppose for each $i=1,\ldots,\ell$, ${M}_i\in \mathring{\OO}_{(p,n)}(d_i)$, where $d_i\in \C$. Then
 $\psi \in \mathring{{\rm KZ}}_{(p,n)}(M^{\otimes})$ (resp., $\mathring{\cS}_{(p,n)}(M^{\otimes})$) if and only if
 $\prod \limits_{1\le i< j\le \ell}(z_i-z_j)^{\frac{p d_id_j}{\kappa}}\psi\in{{\rm KZ}}_{(p,n)}(\Psi(M^{\otimes}))$ (resp., ${\cS}_{(p,n)}(\Psi(M^{\otimes}))$).

  \item Suppose for each $i=1,\ldots,\ell$, ${M}_i\in \mathring{\ov\OO}[q,m]_{(p,n)}(d_i)$, where $d_i\in \C$. Then
 $\psi \in \mathring{\ov{\rm KZ}}[q,m]_{(p,n)}(M^{\otimes})$ (resp., $\mathring{\ov\cS}[q,m]_{(p,n)}(M^{\otimes})$) if and only if
 $\prod \limits_{1\le i< j\le \ell}(z_i-z_j)^{\frac{(q-p) d_id_j}{\kappa}}\psi\in{\ov{\rm KZ}}[q,m]_{(p,n)}(\Psi(M^{\otimes}))$ (resp., ${\ov\cS}[q,m]_{(p,n)}(\Psi(M^{\otimes}))$).

\end{enumerate}
\end{prop}

\begin{thm} \label{iso-KZ} There is a linear isomorphism from the set of singular solutions of the (super) KZ equations of any fixed weight on any tensor product of modules in $\mathring{\ov\OO}[q,m]_{(p,n)}$ to the set of singular solutions of the KZ equations of the corresponding weight on the tensor product of the corresponding modules in $\mathring{\OO}{(r,k)}$ for $r$ and $k$ sufficiently large.
\end{thm}

\begin{proof}
This follows from \propref{sing-sol}, \propref{tr-KZ} and \propref{ncetoce}.
\end{proof}

\section{Cubic Gaudin Hamiltonians for $\gl(m|n)$} \label{sec:cubic}

In this section, we study the cubic Gaudin Hamiltonians for $\gl(m|n)$ on the tensor product of irreducible polynomial modules and their diagonalization. We also discuss the quadratic Gaudin Hamiltonians for $\gl(m|n)$ on the tensor power of the natural module.

\subsection{Settings} \label{sec:CGH}

Let us simplify some notations. Fix $m \in \Zp$ and $n \in \N\cup\{\infty\}$.
Our main object of study in this section is the general linear Lie (super)algebra $\ovcG{[0, m]}_{(0, n)}\cong \gl(m|n)$, and so we denote each symbol ${[0, m]}_{(0,n)}$ by $[m]_n$, e.g.,
$$
\ovcG{[m]}_{n}:=\ovcG{[0, m]}_{(0, n)}, \quad \ov \I{[m]}_n:= \ov \I{[0, m]}_{(0,n)}, \quad \ov \cP{[m]}^+_n:= \ov \cP{[0, m]}^+_{(0,n)}, \quad \textrm{etc.}
$$

Let $k, \ell \in \N$ with $\ell \ge 2$.
Let $u$ be a formal variable and $\partial_u=d/du$.
For $i=1, \ldots, \ell$ and $(z_1, \ldots, z_\ell)\in {\bf X}_\ell$. Consider the matrix $\cL(u):=[\ell_{rs}(u)]_{r, s \in \ov \I[m]_{n}}$, where
$$
\ell_{rs}(u)=\delta_{r,s}\partial_u -(-1)^{2r} \sum_{i=1}^\ell \frac{E_{r,s}^{(i)}}{u-z_i} \qquad \mbox{for $r, s  \in \ov \I[m]_{n}$.}
$$
The supertrace of $\cL(u)^k$ can be written in the form
$$
\text{Str}(\cL(u)^k)=\sum_{j=0}^k S_{kj}(u)\partial_u^{k-j}.
$$
Of particular interest is the series $S_{kk}(u)$ in $u$, whose coefficients are Gaudin Hamiltonians contained in the Gaudin algebra of $\ovcG[m]_n$; cf. \cite[Section 3.2]{MR} (see \cite{FFR, Ta} for the non-super version). The quadratic Gaudin Hamiltonians $\ov H^i[m]_n$'s come from $S_{22}(u)$. In fact,
$$
S_{22}(u)=\sum_{i=1}^\ell \left( \frac{2 \ov H^i[m]_n}{u-z_i} +\frac{\sum_{r,s \in  \ov \I[m]_{n}} (-1)^{2s} E_{r,s}^{(i)} E_{s,r}^{(i)} + \sum_{r \in  \ov \I[m]_{n}} E_r^{(i)}}{(u-z_i)^2} \right).
$$

Now we go one step further and study the diagonalization of $S_{33}(u)$ on the tensor product $\ov L^\otimes:=\ov L_1 \otimes \cdots \otimes \ov L_\ell$, where each $\ov L_j$ is an irreducible polynomial module over $\ovcG[m]_n$. We may express $S_{33}(u)$ as
$$
S_{33}(u)=-\sum_{i=1}^\ell \left(\frac{\cS^i_{1}}{u-z_i} +\frac{\cS^i_{2}}{(u-z_i)^2} +\frac{\cS^i_{3}}{(u-z_i)^3} \right)
$$
for some $\cS^i_{1}, \cS^i_{2}, \cS^i_{3} \in U(\ovcG[m]_n)^{\otimes \ell}$.
Since for $r,s,t  \in  \ov \I[m]_{n}$,
$$E_{t,r}E_{s,t}= (-1)^{4(r+t)(s+t)}(E_{s,t}E_{t,r}-E_{s,r}) + \delta_{r,s} E_{t},$$ and for $i \not=j$,
$$
\frac{1}{(u-z_i)(u-z_j)^2}=\frac{1}{(z_i-z_j)^2(u-z_i)}-\frac{1}{(z_i-z_j)^2(u-z_j)}-\frac{1}{(z_i-z_j)(u-z_j)^2},
$$
we calculate that
\begin{align*}
\cS^i_{1}&= 3 \ovcC^{i}[m]_{n},\\
\cS^i_{2}&= 3\ovcD^{i}[m]_{n} -  2 \sum_{r,s \in  \ov \I[m]_{n}} \sum_{\substack{j=1 \\ j \not=i}}^\ell  \frac{E_{r}^{(i)} E_{s}^{(j)}}{z_i-z_j} +\big(2(m-n)+3 \big) \ov H^i[m]_n, \\
\cS^i_{3}&= \sum_{r,s,t \in  \ov \I[m]_{n}} (-1)^{2(s+t)} E_{r,s}^{(i)} E_{s,t}^{(i)} E_{t,r}^{(i)} + 3 \sum_{r,s \in  \ov \I[m]_{n}} (-1)^{2s} E_{r,s}^{(i)} E_{s,r}^{(i)} + 2\sum_{r \in  \ov \I[m]_{n}} E_r^{(i)}.
\end{align*}
Here
\begin{align}
 \ovcC^{i}[m]_{n}&:=   \sum_{r,s,t \in  \ov \I[m]_{n}} (-1)^{2(s+t)(2(r+t)+1)}  \Bigg[ \sum_{\substack{j,k=1 \\ j \not= k, j \not=i, k \not=i}}^\ell   \frac{E_{r,s}^{(i)} E_{t,r}^{(j)}  E_{s,t}^{(k)}  }{(z_i-z_j)(z_i-z_k)} \label{cGH1}\\
& \hspace{5.5cm} +\sum_{\substack{j=1 \\ j \not=i}}^\ell  \frac{  \big(E_{r,s}^{(i)} E_{t,r}^{(j)}E_{s,t}^{(j)} -E_{r,s}^{(j)}   E_{t,r}^{(i)}  E_{s,t}^{(i)} \big)  }{ (z_i-z_j)^2} \Bigg], \nonumber \\
  \ovcD^{i}[m]_{n}&:= \sum_{r,s,t  \in  \ov \I[m]_{n}}  (-1)^{2(s+t)(2(r+t)+1)}   \sum_{\substack{j=1 \\ j \not=i}}^\ell  \frac{ E_{r,s}^{(j)} E_{t,r}^{(i)}E_{s,t}^{(i)}   }{z_i-z_j}. \label{cGH2}
\end{align}
It is clear that for $i=1, \ldots, \ell$,
each summation of
$\cS^i_{3}$ commutes with the
action of $\ovcG[m]_n$ on the $i$-th factor of the tensor product (see, e.g., \cite[Proposition 4 on p. 49]{Sch})
and
$$
\sum_{j=1, j\not=i}^\ell  \frac{\sum_{r,s \in  \ov \I[m]_{n}} E_{r}^{(i)} E_{s}^{(j)}}{z_i-z_j}=\sum_{j=1, j\not=i}^\ell  \frac{(\sum_{r\in  \ov \I[m]_{n}} E_{r}^{(i)})(\sum_{s \in  \ov \I[m]_{n}}E_{s}^{(j)})}{z_i-z_j},
$$
and hence they act on $\ov L^\otimes$ by fixed scalars. Also, we have seen earlier that $\ov H^i[m]_n$ is diagonalizable on $\ov L^\otimes$ for a generic $(z_1, \ldots, z_\ell)\in \C^\ell$.
Thus the diagonalization problem on $S_{33}(u)$ boils down to the diagonalization of $\ovcC^{i}[m]_{n}$'s and $\ovcD^{i}[m]_{n}$'s on $\ov L^\otimes$.

For $i=1, \ldots, \ell$, we call both $\ovcC^{i}[m]_{n}$ and $  \ovcD^{i}[m]_{n}$ \emph{cubic Gaudin Hamiltonians} for $\ovcG[m]_n$. The \emph{cubic Gaudin Hamiltonians} $\cC^{i}_{n}$ and $\cD^{i}_{n}$ (resp., $\wcC^{i}_\infty$ and $\wcD^{i}_\infty$) for $\cG_{(0, n)}$ (resp., $\wcG_{(0, \infty)}$) are defined in the same way as in \eqref{cGH1} and \eqref{cGH2}, respectively, by using $\I_{(0, n)}$ (resp., $\w \I_{(0, \infty)}$) in place of $ \ov \I[m]_{n}$.

Again, we drop the subscript $n$ if $n=\infty$ and denote ${(0, \infty)}$ by $(0)$. For instance, $\cG_\ze:=\cG_{(0, \infty)}$, $\OO_\ze:=\OO_{(0, \infty)}$, etc.

In the rest of this section, by ${M}^{\otimes}\in \w\OO_\ze$ (resp., $\OO_{(0,n)}$ and $\ov\OO[m]_{n}$), we mean that
$$
 {M}^{\otimes}:= M_1\otimes \cdots \otimes  M_\ell
$$
 for some $M_1, \ldots, M_\ell \in \w\OO_\ze$ (resp., $\OO_{(0,n)}$ and $\ov\OO[m]_{n}$).

\subsection{Isomorphisms between (generalized) eigenspaces} \label{sec:sing-eig}

In view of \remref{p=q=0}, we may identify the action of $\DG_\ze$ (resp., $\G_{(0,n)}$ and $\SG[m]_{n}$) with the action of $\wcG_\ze$ (resp., $\cG_{(0,n)}$ and $\ovcG[m]_{n}$) on $M \in \w\OO_\ze$ (resp., $\OO_{(0, n)}$ and $\ov \OO[m]_n$).

The cubic Gaudin Hamiltonians $\wcC^{i}$ (resp., $\cC^{i}$ and $\ovcC^{i}[m]$) and $\wcD^{i}$ (resp., $\cD^{i}$ and $\ovcD^{i}[m]$) are well-defined operators on ${M}^{\otimes} \in \w \OO_\ze$ (resp., $\OO_\ze$ and $\ov \OO[m]$). To see this, take $\wcC^{i}$ for example.
Suppose $v$ is a weight vector of weight $\mu$ in $M^\otimes \in \w \OO_\ze$. Then $\mu \in \w \Xi_{(0,k)}$ for some $k \in \N$. By \lemref{lem:finite_sum}, it suffices to examine the sum
$$
\sum_{\substack{r,s,t  \in  \w \I_{(0, \infty)} \\ s < r, \, t < r}} (-1)^{2(s+t)(2(r+t)+1)}  E_{r,s}^{(i)} E_{t,r}^{(j)}E_{s,t}^{(j)}v
$$
for $i, j=1, \ldots, \ell$ with $i \not=j$.
It is given by
$$
\sum_{\substack{r,s,t  \in   \w \I_{(0, \infty)} \\ s < r, \, t < r}} (-1)^{2(s+t)}  \big(E_{r,s}^{(i)} E_{s,t}^{(j)}E_{t,r}^{(j)}- E_{r,s}^{(i)} E_{s,r}^{(j)} \big)v,
$$
which is a finite sum as a consequence of \lemref{lem:finite_sum} again.

\begin{lem} \label{g-hom}
The action of the cubic Gaudin Hamiltonians $\wcC^{i}$ (resp., $\cC^{i}_{n}$ and $\ovcC^{i}[m]_{n}$) and $\wcD^{i}$ (resp., $\cD^{i}_{n}$ and $\ovcD^{i}[m]_{n}$) on ${M}^{\otimes} \in \w\OO_\ze$ (resp., $\OO_{(0, n)}$ and $\ov \OO[m]_n$) are $\wcG_{(0)}$(resp., $\cG_{(0, n)}$ and $\ovcG[m]_{n}$)-homomorphisms for each $i=1, \ldots, \ell$.
\end{lem}
\begin{proof}
It is straightforward to verify that $\wcC^{i}  \Delta^{(\ell)} (E_{r,s}) =  \Delta^{(\ell)} (E_{r,s}) \wcC^{i} $  for all $r, s \in \hf \N$, where $\Delta^{(\ell)}: U(\wcG_{(0)}) \to U(\wcG_{(0)})^{\otimes \ell}$ is the $\ell$-fold coproduct on $U(\wcG_{(0)})$. The other cases can be treated similarly.
\end{proof}

We may apply the results in Section \ref{sec:quad-G} and Section \ref{sec:quad-cG} together with the techniques in Section \ref{SD} to relate $\wcC^{i}$, $\cC^{i}$ and $\ovcC^{i}[m]$ (resp., $\wcD^{i}$, $\cD^{i}$ and $\ovcD^{i}[m]$).

\begin{lem} \label{cubic-corresp}
Let $M^{\otimes} \in \w\OO_\ze$,
and let $v  \in M^{\otimes}$ be a weight vector of weight $\mu$.
\begin{enumerate}[\normalfont(i)]
  \item If $\mu \in \Xi_\ze$, then $\wcC^i (v)=\cC^i(v)$ and $\wcD^i (v)=\cD^i(v)$ for each $i=1, \ldots, \ell$.
  \item If $\mu  \in \ov{\Xi}[m]$, then $\wcC^i (v)=\ovcC^i[m](v)$ and $\wcD^i (v)=\ovcD^i[m](v)$ for each $i=1, \ldots, \ell$.
\end{enumerate}
\end{lem}
\begin{proof}
Let us prove (ii).
Fix $i=1, \ldots, \ell$.
Suppose $\mu  \in \ov{\Xi}[m]$. We may assume that $v=v_1 \otimes \cdots \otimes v_\ell$, where $v_j \in M_j$ is a weight vector of weight $\mu_j$ for $j=1, \ldots, \ell$, and $\sum_{j=1}^\ell \mu_j=\mu$.
For $j=1, \ldots, \ell$, $\mu_j \in \ov\Xi[m]$ by \lemref{weight-decomposition2}, and for $r \in \hf \N \backslash  \ov\I[m]$, we have $\mu_j (E_r)=0$ and hence $E_rv_j=0$.

Let $\ov\J_k [m]$ (resp., $\w \J_k$) be the Cartesian product of $k$ copies of $\ov{\I}{[m]}$ (resp., $\hf \N$) for $k=2,3$.
Suppose $(r, s, t )\in \w \J_3$ with $r \notin \ov\I[m]$ and $j=1, \ldots, \ell$. We see that $E_{t,r}v_j=0$, for otherwise the weight of $E_{t,r}v_j$ would lie outside $\w\Xi_\ze$.
Thus for all $(r,s,t) \in \w \J_3 \backslash \ov \J_3 [m]$ and $j,k=1, \ldots, \ell$ such that $i$, $j$ and $k$ are pairwise distinct,
$$
E_{r,s}^{(i)} E_{t,r}^{(j)}  E_{s,t}^{(k)}  v=0.
$$
Moreover, for $(r, s, t )\in \w \J_3$ with $r \notin \ov\I[m]$ and $r \not=s$,
$$
E_{t,r} E_{s,t} v_j=(-1)^{4(r+t)(s+t)} (E_{s,t}E_{t,r}-E_{s,r})v_j= 0.
$$
It follows that for all $(r,s,t) \in \w \J_3 \backslash \ov \J_3 [m]$,
$$E_{r,s}^{(i)} E_{t,r}^{(j)}E_{s,t}^{(j)} v=0 \quad \mbox{and} \quad E_{r,s}^{(j)}   E_{t,r}^{(i)}  E_{s,t}^{(i)} v=0.$$
This proves (ii). The proof of (i) is similar.
\end{proof}

The following proposition is an analogue of \propref{prop:isom} after specialization to $p=q=0$ and can be proved in the same way by \remref{rem:SD}. From now on, we denote $\ovmu[m]:=\ovmu[0,m]$ for $\wmu \in \w{\cP}_\ze^+$.

 \begin{prop} \label{prop:isom_0}
Let $M \in \w\OO_\ze$, and let $\wmu \in \w{\cP}_\ze^+$ be a weight of $M$.

\begin{enumerate}[\normalfont(i)]
\item There exists $A_\mu \in U (\tilde{\mf l}_\ze)$ such that the map
${\ft}^{\wmu}_\ze: M_{\wmu}^{\rm sing}  \to T_\ze (M)_{\mu}^{\rm sing}$,
defined by ${\ft}^{\wmu}_\ze(v) = A_\mu v$ for $v \in M_{\wmu}^{\rm sing}$,
is a linear isomorphism.
Moreover, there exists $B_\mu\in U (\tilde{\mf l}_\ze)$
such that the inverse of ${\ft}^{\wmu}_\ze$ is given by
$\big({\ft}^{\wmu}_\ze \big)^{-1}(w)=B_\mu w$ for $w \in T_\ze(M)_{\mu}^{\rm sing}$.

\item There exists $\bar{A}_\mu \in U (\tilde{\mf l}_\ze)$ such that the map
$\ov{\ft}^{\wmu}_{[m]} : M_{\wmu}^{\rm sing}  \to \ov{T}_{[m]}(M)_{\ovmu[m]}^{\rm sing}$,
defined by $\ov{\ft}^{\wmu}_{[m]} (v) = \bar{A}_\mu v$ for $v \in M_{\wmu}^{\rm sing}$,
is a linear isomorphism.
Moreover, there exists $\ov B_\mu\in U (\tilde{\mf l}_\ze)$
such that the inverse of $\ov{\ft}^{\wmu}_{[m]}$ is given by
$\big(\ov{\ft}^{\wmu}_{[m]}\big)^{-1}(w)=\ov B_{\mu} w$
for $w \in \ov{T}_{[m]}(M)_{\ovmu[m]}^{\rm sing}$.
\end{enumerate}
\end{prop}

\begin{prop} \label{prop:C=}
For $\w{M}^{\otimes}:=\w M_1\otimes \cdots \otimes \w M_\ell \in \w\OO_\ze$, let  ${M}^{\otimes}= T_\ze( \w{M}^{\otimes})$ and $ \ov {M}^{\otimes}=\ov T_{[m]}( \w{M}^{\otimes})$. For $\wmu \in \w{\cP}_\ze^+$ and each $i=1, \ldots, \ell$, we have:
\begin{enumerate}[\normalfont(i)]
\item The action of $\cC^i$ (resp., $\cD^i$ and $H^i_\ze$)  on $({M}^{\otimes})^{\rm sing}_{\mu}$ equals ${\ft}^{\wmu}_\ze \circ\wcC^i\circ ({\ft}^{\wmu}_\ze)^{-1}$ (resp., ${\ft}^{\wmu}_\ze \circ\wcD^i\circ ({\ft}^{\wmu}_\ze)^{-1}$ and ${\ft}^{\wmu}_\ze \circ \w H^i_{\ze} \circ ({\ft}^{\wmu}_\ze)^{-1}$).
\item The action of $\ovcC^i[m]$ (resp., $\ovcD^i[m]$ and $\ov H^i[m]$) on $(\ov{M}^{\otimes})^{\rm sing}_{\ov\mu[m]}$ equals $\ov{\ft}_{[m]}^{\wmu}\circ\wcC^i\circ (\ov{\ft}_{[m]}^{\wmu})^{-1}$ (resp., $\ov{\ft}_{[m]}^{\wmu}\circ\wcD^i\circ (\ov{\ft}_{[m]}^{\wmu})^{-1}$ and $\ov{\ft}_{[m]}^{\wmu}\circ \w H^i_{\ze}  \circ (\ov{\ft}_{[m]}^{\wmu})^{-1}$).
\end{enumerate}
\end{prop}
\begin{proof}
This can be proved in the same way as in \propref{prop:H=} by making use of \lemref{cubic-corresp} and \propref{prop:isom_0}.
\end{proof}

As a consequence of \propref{prop:C=}, we have the following theorem.

\begin{thm} \label{thm:eigenvector-cubic}
For $\w{M}^{\otimes}:=\w M_1\otimes \cdots \otimes \w M_\ell \in \w\OO_\ze$, let  ${M}^{\otimes}= T_\ze( \w{M}^{\otimes})$ and $ \ov {M}^{\otimes}=\ov T_{[m]}( \w{M}^{\otimes})$. For $\wmu \in \w{\cP}_\ze^+$ and each $i=1, \ldots, \ell$, we have:

\begin{enumerate}[\normalfont(i)]
\item The actions of $\wcC^i$ (resp., $\wcD^i$) on $(\w{M}^{\otimes})_{\wmu}^{\rm sing}$ and $\cC^i$ (resp., $\cD^i$) on $({M}^{\otimes})^{\rm sing}_{\mu}$ have the same spectrum, and the map ${\ft}^{\wmu}_{\ze}$ gives a linear isomorphism from the (generalized) eigenspace of $\wcC^i$ (resp., $\wcD^i$) on $(\w{M}^{\otimes})_{\wmu}^{\rm sing}$ to the (generalized) eigenspace of $\cC^i$ (resp., $\cD^i$) on $({M}^{\otimes})^{\rm sing}_{\mu}$ for each common eigenvalue $c$.

\item The actions of $\wcC^i$ (resp., $\wcD^i$) on $(\w{M}^{\otimes})_{\wmu}^{\rm sing}$ and $\ovcC^i[m]$ (resp., $\ovcD^i[m]$) on $(\ov{M}^{\otimes})^{\rm sing}_{\ovmu[m]}$ have the same spectrum, and the map $\ov{\ft}^{\wmu}_{[m]}$ gives a linear isomorphism from the (generalized) eigenspace of $\wcC^i$ (resp., $\wcD^i$) on $(\w{M}^{\otimes})_{\wmu}^{\rm sing}$ to the (generalized) eigenspace of $\ovcC^i[m]$ (resp., $\ovcD^i[m]$) on $(\ov{M}^{\otimes})^{\rm sing}_{\ovmu[m]}$ for each common eigenvalue $c$.
  \end{enumerate}
As a consequence, for each $i$, $\ovcC^i[m]$ (resp., $\ovcD^i[m]$) is diagonalizable on $(\ov {M}^{\otimes})^{\rm sing}_{\ovmu[m]}$ if and only if $\cC^i$ (resp., $\cD^i$) is diagonalizable on $({M}^{\otimes})^{\rm sing}_{\mu}$. In this case, they have the same spectrum.
\end{thm}

\begin{rem}
The Jordan canonical forms of $\wcC^i$ (resp., $\wcD^i$) on $(\w{M}^{\otimes})_{\wmu}^{\rm sing}$, $\cC^i$ (resp., $\cD^i$) on $({M}^{\otimes})^{\rm sing}_{\mu}$ and $\ovcC^i[m]$ (resp., $\ovcD^i[m]$) on $(\ov{M}^{\otimes})^{\rm sing}_{\ovmu[m]}$ are the same for each fixed $i$.
\end{rem}

\subsection{Diagonalization and correspondences between eigenbases}
Let $\cP$ denote the set of all partitions. For $n \in \N \cup \{\infty\}$. Define
\begin{align*}
Q_{(0,n)}&=\setc*{\sum_{i=1}^n \la'_i \ep_{i-\hf}+d \La_0 \in \fh_{(0, n)}^* }{ \la \in  \cP, \, d \in \N }, \\
\ov{Q}[m]_n&=\setc*{\sum_{i=1}^{m} \la_{i}\ep_{i}+ \sum_{j =1}^n \langle \la'_{j}-m \rangle\ep_{j-\hf}+ d\La_0  \in \ovfh[m]_{n}^*}{ \la  \in \cP, \, d \in \N }.
\end{align*}
In fact, $Q_{(0,\infty)}$ and $\ov{Q}[m]_\infty$ are respectively the sets $Q$ and $Q[q,m]$, defined in \eqref{Qqm}, after the specialization $p=q=0$.
We also have the bijection
\begin{equation}\label{Q-0}
   Q_{(0,\infty)} \leftrightarrow \ov{Q}[m]_\infty.
\end{equation}

Define
\begin{align}
\cQ_{(0,n)}&=\setc*{\sum_{i=1}^n \la'_i \ep_{i-\hf} \in {\bf h}_{(0, n)}^* }{ \la \in  \cP  }, \label{Q0} \\
\ov{\cQ}[m]_n&=\setc*{\sum_{i=1}^{m} \la_{i}\ep_{i}+ \sum_{j =1}^n \langle \la'_{j}-m \rangle\ep_{j-\hf}   \in \ov{\bf h}[m]_{n}^*}{ \la  \in \cP}. \label{Qm}
\end{align}
If $n \in \N$, then $\xi \in \cQ_{(0,n)}$ is a dominant integral weight, and $L(\ovcG[m]_n, \ov{\xi})$ is an irreducible polynomial $\ovcG[m]_n$-module for $\ov \xi \in \ov{\cQ}[m]_n$.

By \remref{p=q=0}, we may declare $\Lambda_0$ to be $0$ and identify $Q_{(0,n)}$ with $\cQ_{(0,n)}$ and $\ov{Q}[m]_n$ with $\ov{\cQ}[m]_n$. We may also make the identifications
$L(\G_{(0,n)}, \xi)=L(\cG_{(0,n)}, \xi)$ for $\xi \in \cQ_{(0,n)}$ and $L(\SG[m]_n, \ov \xi)=L(\ovcG[m]_n, \ov{\xi})$ for $\ov \xi \in \ov{\cQ}[m]_n$.

We drop the subscript $n$ if $n=\infty$. According to \eqref{Q-0}, every weight $\xi \in \cQ_\ze$ corresponds to a weight $\ov\xi \in \ov{\cQ}[m]$, and vice versa.

\begin{cor}\label{cor:cubic-diag}
Let $\un{\xi}  \in  \cQ_\ze^{\times \ell}$. For a generic $(z_1, \ldots, z_\ell)\in \C^\ell$, the cubic Gaudin Hamiltonians
$\cC^1,\ldots, \cC^{\ell}$
and
$\cD^1,\ldots,\cD^{\ell}$
are simultaneously diagonalizable on the space $L(\cG_{(0)}, \un{\xi})^{\rm sing}$.
\end{cor}
\begin{proof}
This follows from \propref{Gaudin-alg} and an argument similar to the proof of \corref{cor:quad-diag-central}.
\end{proof}

\begin{thm}\label{thm:cubic-diag-central}
Let $\un{\ov \xi} \in  \ov \cQ[m]^{\times \ell}$, and let $\un{\xi} \in  \cQ_\ze^{\times \ell}$ be the corresponding sequence of weights under \eqref{Q-0}.
Let $\ovmu \in  \ov \cQ[m]$ and $\mu \in \cQ_\ze$ the corresponding weight. For a generic $(z_1, \ldots, z_\ell)\in \C^\ell$, we have:

\begin{enumerate}[\normalfont(i)]

\item The cubic Gaudin Hamiltonians
$\ovcC^1[m],\ldots,\ovcC^\ell[m]$
and
$\ovcD^1[m],\ldots,\ovcD^\ell[m]$
 are simultaneously diagonalizable on $L(\ovcG[m],\un{\ov{\xi}})^{\rm sing}_{\ovmu}$.

\item Every joint eigenbasis of the Hamiltonians
$\ovcC^1[m],\ldots,\ovcC^\ell[m]$
 (resp., $\ovcD^1[m],\ldots,\ovcD^\ell[m]$) on $L(\ovcG[m],\un{\ov{\xi}})^{\rm sing}_{\ovmu}$ can be obtained from some joint eigenbasis of
 $\cC^1,\ldots, \cC^{\ell}$
  (resp., $\cD^1,\ldots, \cD^{\ell}$) on $L(\cG_\ze, \un{{\xi}})^{\rm sing}_\mu$, and vice versa.

\item For each $i=1,\ldots,\ell$, the actions of $\ovcC^i[m]$
    (resp., $\ovcD^i[m]$) on $L(\ovcG[m],\un{\ov{\xi}})^{\rm sing}_{\ovmu}$ and $\cC^i$ (resp., $\cD^i$)
     on $L(\cG_\ze, \un{{\xi}})^{\rm sing}_\mu$ have the same spectrum.

\end{enumerate}
\end{thm}
\begin{proof}
This follows from \thmref{thm:eigenvector-cubic} and \corref{cor:cubic-diag}.
\end{proof}

Let $n\in \N$. To establish the diagonalization of the Hamiltonians $\ovcC^i[m]_{n}$ (resp., $\ovcD^i[m]_{n}$) on the space $L(\ovcG[m]_{n}, \un{\ov\xi})$ for $\ov\xi_1, \ldots, \ov\xi_\ell \in \ov\cQ[m]_{n}$, we just need to establish the diagonalization of $\ovcC^i[m]_{n}$ (resp., $\ovcD^i[m]_{n}$) on the subspace $L(\ovcG[m]_{n}, \un{\ov\xi})^{\rm sing}$ by \lemref{g-hom}.

\begin{thm} \label{thm:cubic-diag}
Let $n\in \N$, $\un{\ov \xi} \in \ov\cQ[m]_{n}^{\times \ell}$ and $\ovmu \in \ov\cQ[m]_{n}$. For a generic $(z_1, \ldots, z_\ell)\in \C^\ell$, the cubic Gaudin Hamiltonians $\ovcC^1[m]_{n},\ldots,\ovcC^\ell[m]_{n}$
and
$\ovcD^1[m]_{n},\ldots,\ovcD^\ell[m]_{n}$
 are simultaneously diagonalizable on the space $L(\ovcG[m]_{n}, \un{\ov\xi})^{\rm sing}_{\ovmu}$,
and each joint eigenbasis of the Hamiltonians can be obtained from some joint eigenbasis of the quadratic Gaudin Hamiltonians
 $H^1_{(0,k)},\ldots,H^\ell_{(0,k)}$ on the corresponding singular weight space of the tensor product of the corresponding finite-dimensional irreducible $\cG_{(0,k)}$-modules for $k$ sufficiently large.
\end{thm}

\begin{proof}
The Gaudin Hamiltonians $H^i_{(0,k)}$, $\cC^i_k$ and $\cD^i_k$ ($i=1,\ldots,\ell$) are simultaneously diagonalizable for $k \in \N$ and a generic $(z_1, \ldots, z_\ell)\in \C^\ell$.
The statement then follows from \thmref{thm:cubic-diag-central} and an argument similar to the proof of \thmref{thm:quad-diag}.
\end{proof}

We now return to the study of the quadratic Gaudin Hamiltonians for $\ovcG[m]_n$ on the tensor power of the natural module $\C^{m|n} \cong L(\ovcG[m]_n,\ep_1)$.

\begin{thm}\label{thm:cyclic}
Let $n \in \N$ and $\un{\ov \xi}=(\ep_1, \ldots, \ep_1)  \in \ov\cQ[m]_{n}^{\times \ell}$. For each $(z_1, \ldots, z_\ell) \in {\bf X}_\ell$, the action of the algebra generated by the quadratic Gaudin Hamiltonians $\ov H^1[m]_n,\ldots, \ov H^{\ell}[m]_n$ on the space $L(\ovcG[m]_n,\un{\ov\xi})^{\rm sing}$ contains a cyclic vector.
\end{thm}

\begin{proof}
Since $L(\ovcG[m]_n,\un{\ov\xi})^{\rm sing}$ is finite-dimensional, $L(\ovcG[m]_n,\un{\ov\xi})^{\rm sing}=\bigoplus_{\ov\mu}L(\ovcG[m]_n,\un{\ov\xi})^{\rm sing}_{\ov\mu}$ is a finite direct sum. Applying \propref{prop:isom_0} as well as the arguments used above, we can choose $k$ sufficiently large that each weight $\mu$ that corresponds to the singular weight $\ovmu$ of $L(\ovcG[m]_n,\un{\ov\xi})$ is a singular weight of $L(\cG_{(0,k)},\un{\xi})$, where $\un{\xi}:=(\ep_\hf, \ldots, \ep_\hf) \in  \cQ_{(0,k)}^{\times \ell}$.
Denote $M=\bigoplus_\mu L(\cG_{(0,k)},\un{\xi})^{\rm sing}_\mu$, where $\mu$ runs over all singular weights corresponding to singular weights of $L(\ovcG[m]_n,\un{\ov\xi})$.
Let $\hat{\ft}=\bigoplus_{\wmu} {\ft}^{\wmu}_\ze \circ (\ov{\ft}^{\wmu}_{[m]})^{-1}\,:\,L(\ovcG[m]_n,\un{\ov\xi})^{\rm sing}\longrightarrow M$, where $\wmu$ runs over all singular weights corresponding to singular weights of $L(\ovcG[m]_n,\un{\ov\xi})$.
\propref{prop:C=} implies that the action of $\ov H^i[m]_n$ on $L(\ovcG[m]_n,\un{\ov\xi})^{\rm sing}$ equals $\hat{\ft}^{-1}\circ H^i_{(0,k)}\circ \hat{\ft}$ for all $i$.
By  \cite[Theorem 3.2]{MTV10} and \cite[Main Theorem]{Ryb}, we see that the action of the algebra generated by the Hamiltonians $H^1_{(0,k)},\ldots, H^{\ell}_{(0,k)}$ on the space $M$ contains a cyclic vector and hence the theorem follows.
\end{proof}

\begin{cor}
Let $n \in \N$ and $\un{\ov \xi}=(\ep_1, \ldots, \ep_1)  \in \ov\cQ[m]_{n}^{\times \ell}$. For a generic $(z_1, \ldots, z_\ell)\in \C^\ell$, the action of the algebra generated by the Hamiltonians $\ov H^1[m]_n,\ldots, \ov H^{\ell}[m]_n$ is simultaneously diagonalizable on the space $L(\ovcG[m]_n,\un{\ov\xi})^{\rm sing}$ with a simple joint spectrum.
\end{cor}

\begin{proof}
This follows from \thmref{thm:quad-diag} and \thmref{thm:cyclic}.
\end{proof}

The above corollary matches the result obtained earlier by Mukhin, Tarasov, and Young using the Bethe ansatz method (cf. \cite[Theorem 5.2]{MVY}).

\vskip 0.5cm
\noindent{\bf Acknowledgments.}
The first author was partially supported by NSFC (Grant No. 12161090).
The second author was partially supported by MoST grant 110-2115-M-006-006 of Taiwan.
Part of this research was done during the visit of the third author to Yunnan University. The third author was partially supported by MoST grant 109-2115-M-006-019-MY3 and NSTC grant 112-2115-M-006-015-MY2 of Taiwan and he thanks Yunnan University for hospitality and support.
The authors are pleased to thank the referee for useful suggestions.

\end{document}